\documentclass[11pt]{article}
\usepackage{latexsym,amsfonts,amssymb,amsmath,amsthm}
\usepackage{graphicx}
\usepackage{cite}
\usepackage[usenames,dvipsnames]{color}
\usepackage{ulem}
\usepackage[colorlinks=true]{hyperref}
\hypersetup{urlcolor=blue, citecolor=red}

\newcommand {\rd}{\color{red}}

\begin{document}
\setlength{\baselineskip}{16pt}

\parindent 0.5cm
\evensidemargin 0cm \oddsidemargin 0cm \topmargin 0cm \textheight
22cm \textwidth 16cm \footskip 2cm \headsep 0cm

\newtheorem{theorem}{Theorem}[section]
\newtheorem{lemma}[theorem]{Lemma}
\newtheorem{proposition}[theorem]{Proposition}
\newtheorem{definition}{Definition}[section]
\newtheorem{example}{Example}[section]
\newtheorem{corollary}[theorem]{Corollary}

\newtheorem{remark}{Remark}[section]
\newtheorem{property}[theorem]{Property}
\numberwithin{equation}{section}
\newtheorem{mainthm}{Theorem}
\newtheorem{mainlem}{Lemma}

\numberwithin{equation}{section}

\def\p{\partial}
\def\I{\textit}
\def\R{\mathbb R}
\def\C{\mathbb C}
\def\u{\underline}
\def\l{\lambda}
\def\a{\alpha}
\def\O{\Omega}
\def\e{\epsilon}
\def\ls{\lambda^*}
\def\D{\displaystyle}
\def\wyx{ \frac{w(y,t)}{w(x,t)}}
\def\imp{\Rightarrow}
\def\tE{\tilde E}
\def\tX{\tilde X}
\def\tH{\tilde H}
\def\tu{\tilde u}
\def\d{\mathcal D}
\def\aa{\mathcal A}
\def\DH{\mathcal D(\tH)}
\def\bE{\bar E}
\def\bH{\bar H}
\def\M{\mathcal M}
\renewcommand{\labelenumi}{(\arabic{enumi})}

\def\disp{\displaystyle}
\def\undertex#1{$\underline{\hbox{#1}}$}
\def\card{\mathop{\hbox{card}}}
\def\sgn{\mathop{\hbox{sgn}}}
\def\exp{\mathop{\hbox{exp}}}
\def\OFP{(\Omega,{\cal F},\PP)}
\newcommand\JM{Mierczy\'nski}
\newcommand\RR{\ensuremath{\mathbb{R}}}
\newcommand\CC{\ensuremath{\mathbb{C}}}
\newcommand\QQ{\ensuremath{\mathbb{Q}}}
\newcommand\ZZ{\ensuremath{\mathbb{Z}}}
\newcommand\NN{\ensuremath{\mathbb{N}}}
\newcommand\PP{\ensuremath{\mathbb{P}}}
\newcommand\abs[1]{\ensuremath{\lvert#1\rvert}}

\newcommand\normf[1]{\ensuremath{\lVert#1\rVert_{f}}}
\newcommand\normfRb[1]{\ensuremath{\lVert#1\rVert_{f,R_b}}}
\newcommand\normfRbone[1]{\ensuremath{\lVert#1\rVert_{f, R_{b_1}}}}
\newcommand\normfRbtwo[1]{\ensuremath{\lVert#1\rVert_{f,R_{b_2}}}}
\newcommand\normtwo[1]{\ensuremath{\lVert#1\rVert_{2}}}
\newcommand\norminfty[1]{\ensuremath{\lVert#1\rVert_{\infty}}}

\title{Dynamics of a class of time-period strongly 2-cooperative system: integer-valued Lyapunov function and embedding property of limit sets}

\author {
\\
Mengmeng Gao and Dun Zhou\thanks{Corresponding author, Email: zhoudun@njust.edu.cn.  Partially supported by NSF of China No.11971232, 12331006, 12071217.}\\
School of Mathematics and Statistics\\
 Nanjing University of Science and Technology
\\ Nanjing, Jiangsu, 210094, P. R. China
\\
}
\date{}

\maketitle
% insert the table of contents
%\tableofcontents

%---------------------SECTION DIVIDE LINE---------------------------
\begin{abstract}
We construct an integer-valued Lyapunov function $\sigma(\cdot)$ for generalized negative cyclic feedback system; and prove that  $\sigma(\cdot)$ on any $\omega$-limit set which generated 
by Poincar\'{e} mapping of bounded solution of such strongly $2$-cooperative system is constant. Therefore, the $\omega$-limit can be continuously embedded into a compact subset of a two dimensional plane.  Finally, a dissipative condition is given to ensure that all orbits of such system are bounded.
\end{abstract}
\section{Introduction}
The concept of monotone dynamical systems(MDS) was introduced by M. W. Hirsch, H. Matano in the 1980s (\cite{hirsch1982,hirsch1984,hirsch1988a,hirsch19886,hirsch1989,hirsch1990,Matano84,Matano86,Matano87}); and has wide range of applications in biomathematics and control theory. From the view of geometry, the MDS has a special structure-closed convex cone, with the help of this special structure, ``most'' bounded trajectories will approach a set of equilibria, that is, so called generic quasi-convergence. Among others, cooperative systems are typical MDS. Taking a finite-dimensional system as an example:
\begin{equation}\label{cooperative}
  \dot{x}=f(x),\quad x\in D
\end{equation}  
where $D\subset \mathbb{R}^n$ is a non-empty, open convex set, $f\in C^1(\Omega)$. Then, \eqref{cooperative} is a cooperative if and only if
\begin{equation}\label{cooperative1}
  \frac{\partial f_i}{\partial x_j}(x)\geq 0, \quad i\neq j, x\in D,
\end{equation}
holds true. Moreover, if the Jacobi matrix of $f$ at every point $x\in D$ is irreducible, then \eqref{cooperative}- is a strongly cooperative system(see \cite{Smi95}).

Nevertheless, there are a lot of systems arising in applied science do not belong MDS, for instance, the negative cyclic feedback system(see \cite{MP1990_sub}). In this context, a generalization of MDS was proposed, Sanchez in \cite{Sanchez} introduced the invariant cones of rank $k$. A non-emepty closed set $C\subset\mathbb{R}^n$ is cone of rank $k$, if: (i) $x\in C$, $\alpha \in\mathbb{R}^n\Rightarrow \alpha x\in C$; (ii) $\max\{\dim W: C\supset W \text{ linear subspace}\}=k$. It is obvious that, MDS are systems with cones of rank $1$. By using cones of rank $2$, Sanchez projected part of the dynamics into planes, and then deduce the Poincar\'{e}-Bendixson property for some orbits, that is, some compact omega-limit sets without equilibrium points are just a periodic orbit. Later, this property was generalized to generic Poincar\'{e}-Bendixson theory for systems with cones of rank $2$ by Feng et al. \cite{FWW}.

In the current paper, we focus on a time periodic system in the following form:
\begin{equation}\label{system 1.1}
\begin{split}
\dot{x}_1 &=f_1(t,x_1,x_2,x_n),\\
\dot{x}_i &=f_i(t,x_{i-1},x_i,x_{i+1}),\quad 2\leq i\leq n-1.\\
\dot{x}_n &=f_n(t,x_{n-1},x_n ,x_1),\\
%\tag{1}
\end{split}
\end{equation}
with the nonlinear term satisfies
\begin{equation}\label{assume system 1.1}
 \begin{split}
\delta_{i}\frac{\partial f_{i}}{\partial x_{i-1}}(t,x)> 0&,\quad 1\leq i\leq n,\\
\delta_{i+1}\frac{\partial f_{i}}{\partial x_{i+1}}(t,x)\geq 0&,\quad 2\leq i\leq n-1,\\
  \end{split}
\end{equation}
where $\delta_{i}\in \{-1,1\},\ x_0=x_n,\ x_{n+1}=x_1$,  $f=(f_1,f_2,\cdots,f_n)$ is a $C^1$-function defined on $\mathbb{R}\times \Omega\subset\mathbb{R}  \times \mathbb{R}^{n}$ and $\Omega$ is a non-empty, open and convex set defined on $\mathbb{R}^n$, and there exists $T>0$
 such that $f(t+T,\cdot) \equiv f(t,\cdot)$.  Such structure widely exists in biological, control models: biological oscillators, neurological systems, cell regulation, enzyme reactions, as well as gene transcription \cite{Fer,Hasting1977,EL,TCM08}. 

For instance, consider the following four dimensional system:
\begin{equation}
\begin{aligned}
& \dot{x}_1=a_1-a_2 x_1 x_4, \\
& \dot{x}_2=a_3 x_1-a_4 x_2, \\
& \dot{x}_3=a_5 x_2-a_6 x_3, \\
& \dot{x}_4=a_7 x_3-a_8 x_1 x_4,
\end{aligned}
\end{equation}
where $a_i>0, i=1,\cdots,8$ are constant. It is an integral feedback model named antithetic controllers(see \cite{BGK}). The above system characterizes the closed-loop interconnection of an antithetic controller (reprented by the
variables $x_1$ and $x_4$) and a simple two-dimensional linear system (represented by the variables $x_2$ and
$x_3$). 

\eqref{system 1.1}+\eqref{assume system 1.1} is a generalization of the following cyclic feedback system that introduced by J. Mallet-Paret and H. Smith in \cite{MP1990_sub}  
\begin{equation}\label{MS-1}
 \begin{aligned}
\dot{x}_1&=f_{1}(x_{1},x_{n}), \\
\dot{x}_i&=f_{i}(x_{i-1},x_{i}),\quad 2 \leq i\leq n-1. \\
\dot{x}_n&=f_{n}(x_{n-1},x_{n})\\
  \end{aligned}
 \end{equation}
and satisfies
\begin{equation}\label{MS-2}
  \delta_{i}\frac{\partial f_{i}}{\partial x_{i-1}}(x_{i},x_{i-1})>0,\quad 1\leq i\leq n
\end{equation}
where $\delta_{i}\in \{-1,1\}$.

Let $\Delta=\delta_1\cdots\delta_n$, $x_i= u_{i}x_{i},u_{i}\in\{+1,-1\},1\leq i\leq n$, then \eqref{system 
1.1}+\eqref{assume system 1.1} becomes a new feedback system (see Section \ref{transformation}) that satisfies 
\begin{equation}\label{assume simple system 1.1}
  \begin{aligned}
    \delta_{1}\frac{\partial f_1}{\partial x_n} & >  0,\  \delta_{1}\frac{\partial f_n}{\partial x_1}\geq  0,\\
    \frac{\partial f_i}{\partial x_{i-1}} & >0,\quad 2\leq i \leq n,\\
    \frac{\partial f_i}{\partial x_{i+1}} & \geq 0,\quad 2\leq i \leq n-1.
  \end{aligned}
\end{equation}
Therefore, we can always assume, without loss of generality, that \eqref{system 1.1}+\eqref{assume system 1.1} is in the form of \eqref{system 1.1}+\eqref{assume simple system 1.1}. Then, if $\delta_1=\Delta= +1$, \eqref{system 1.1}+\eqref{assume simple system 1.1} is a cooperative system with positive feedback; if $\delta_1=\Delta= -1$, it is a non-cooperative system with negative feedback.

In a positive feedback system, the final product of a behavior will lead to more behavior occurring in the feedback loop, amplifying the initial action. Low frequency oscillators, birthing processes, and the expression of premembrane protein are positive feedback systems(see \cite{AL,Haraoubia}). A negative feedback system is a system that, when subjected to external stimuli or disturbances, the system will suppress or weaken this external excitation or interference through negative feedback, thereby achieving a new equilibrium state of the system. Typical negative feedback models include interspecific competition and cooperation systems, Lotka Volterra models, negative feedback loops of enterobacterium transcriptional replicators, and so on (see \cite{HN2018,MH06}). These negative feedback loops are considered the main reason for the existence of biological oscillations.

In the case  $\delta_1= +1$(positive feedback), \eqref{system 1.1}+\eqref{assume simple system 1.1} is 
a strongly monotone dynamic system (see \cite{Smi95}). And hence, if \eqref{system 1.1}+\eqref{assume simple system 1.1} is an autonomous system ($f$ is independent of $t$ \cite{ST91}), the $\omega$-limit set generated by ``most'' bounded trajectories is an equilibrium point. In fact, Poincar\'{e}-Bendixson theorem is established, that is, the $\omega$-limit set of any bounded solution of \eqref{system 1.1}+\eqref{assume simple system 1.1} contains no equilibrium should be a periodic orbit(see \cite{Per}); moreover, generic Morse-Smale property was also obtained(see \cite{fusco,Per}). If \eqref{system 
1.1}+\eqref{assume simple system 1.1} is a time $T$-periodic system(i.e. $f(t+T,\cdot)\equiv f(t,\cdot)$), ``most'' bounded orbits of the system 
will converge to a linearly stable $kT$-periodic orbit, where $k$ is a positive integer(see \cite{Pol92}). 

In the case  $\delta_1= -1$(negative feedback), \eqref{system 1.1}+\eqref{assume simple system 1.1} is no 
longer a MDS from the view of cone. In autonomous case ($f=(f_1,\cdots,f_n)^T$ independent of $t$), J. Mallet-Paret and H. Smith in \cite{MP1990_sub},  obtained the Poincar\'{e}-Bendixson theorem for \eqref{MS-1}+\eqref{MS-2}. Later, the Morse decomposition of global attractor was studied by T. Gedeon in \cite{Gedeon}, and complicated dynamics take place within some Morse set. If  \eqref{MS-1}+\eqref{MS-2} admits a unique equilibrium, then it is a globally asymptotically under suitable assumptions (see \cite{WLS}). Elkhader in \cite{elkhader1992_sub} extended the results in \eqref{system 1.1}+\eqref{assume simple system 1.1} to a more general form, that is, $f_1=f_1(x_1,x_n)$, $f_n=f_n(x_{n-1},x_n)$ in \eqref{system 1.1}+\eqref{assume system 1.1}, under the assumption that $f$ is a $C^{n-1}$ function on $\Omega$. Recently, by introducing the notion of $k$-cooperative dynamical systems(see Section \ref{k-cooperative}), \eqref{system 1.1}+\eqref{assume simple system 1.1} is a strongly $2$-cooperative system; and hence, Poincar\'{e}-Bendixson theorem is proved for \eqref{system 1.1}+\eqref{assume simple system 1.1} (see \cite{MS,weiss2021_sub}) by projecting the $\omega$-limit set of bounded orbit into a suitable $2$-dimensional space. 

In practical problems, the nonlinear term $f$ in \eqref{system 1.1}+\eqref{assume simple system 1.1} often depends on $t$, that is, \eqref{system 1.1}+\eqref{assume simple system 1.1} is a non-autonomous system. A natural question is how to describe the long time behavior of the solution of \eqref{system 1.1}+\eqref{assume simple system 1.1}, and can we project the $\omega$-limit set into a two dimensional plane? We try to answer this question and give some characterization about the 
global dynamics of \eqref{system 1.1}+\eqref{assume simple system 1.1} when $f$ of $t$ is $T$-period. 

We first generalize the important tool called the integer-valued Lyapunov function $\sigma(\cdot)$ (see Definition \ref{L-function}) in \cite{elkhader1992_sub,MP1990_sub} to \eqref{system 1.1}+\eqref{assume simple system 1.1}(see Theorem \ref{linear system large-constant}). $\sigma(\cdot)$ is only on an open and dense subset $\Lambda$ of $\mathbb{R}^n$ on which it is also continuous. The deductions in Theorem \ref{linear system large-constant} are different from that in \cite[Proposition 2.1]{elkhader1992_sub} or \cite[Propsotion 1.1]{MP1990_sub}; and moreover, the form of $\sigma(\cdot)$ and it's decreasing properties are uniquely determined by each other (see Proposition \ref{th-monotone of NmM}).

As an application, we then use $\sigma(\cdot)$ to reduce the dynamics on $\omega$-limit set of any bounded solution of \eqref{system 1.1}+\eqref{assume simple system 1.1} to the dynamics on a compact subset of $\mathbb{R}^2$. Precisely speaking, assume that  $\phi(t,x^0)$ is 
a forward bounded solution of \eqref{system 1.1}+\eqref{assume simple system 1.1} with an initial value of $x^0$. Let $P$ be the associated  Poincar\'{e} map  and $\omega^{p}(x^0)$ be the $\omega$-limit set. Then 

\begin{itemize}
  \item (see Theorem \ref{changzhixing-th}) For any two point $x,y \in \omega^{p}(x^{0}) \ (x\neq y)$ and $t\in \mathbb{R}$, there is 
  \begin{equation*} 
\sigma(\phi(t,x)-\phi(t,y))\equiv constant;
  \end{equation*}
   \item (see Theorem \ref{main}) $\omega^p(x^0)$ can be continuous embedded into a compact subset of two-dimensional plane. 
\end{itemize}

We give some remarks in the following 

 1)  From the view skew-product flows, Theorem \ref{main} means that the dynamics on $\cup_{t\in[0,T)}\omega^p(\phi(t,\\ x^0))$ can be viewed as an $T$-periodically forced flow on a compact subset of $\mathbb{R}^2$. And it should be noted that, in general, dynamics of two-dimensional discrete mappings still can be very complicated, for instance, the existence of Smale's horseshoe mapping.  Nevertheless, in some specific models, the dynamics may become simpler due to the constraints of some conditions (see Section \ref{1-2 cooperative}).

2) Our results can be viewed as a partial generalization of autonomous case in \cite{MS,weiss2021_sub} to time periodic case. In \cite{MS,weiss2021_sub}, nested invariant cones were frequently used to reduce the complexity of dynamics on $\omega$-limit. While in this paper, we give a delicate characterization of integer-valued Lyapunov function $\sigma(\cdot)$; and by using $\sigma(\cdot)$, we obtain the constancy property of $\sigma(\cdot)$ on $\omega^p(x_0)$ and finally achieve our goal. In fact, based on Theorem \ref{linear system large-constant}, we can further establish the Floquet theory for linearized equation of \eqref{system 
1.1}+\eqref{assume system 1.1}, then construct nested invariant cones, and finally obtain the structural stability of \eqref{system 
1.1}+\eqref{assume system 1.1} in autonomous case. Moreover, the continuous embedding can be then improved to Lipschitz embedding. All these will be included in our forthcoming paper.

3) The discovery of integer-valued Lyapunov function can be traced back to the characterization of oscillating matrix eigenvectors (see \cite[p.105]{Ga}).
 Similar functions widely exist in many other mathematical models developed in physics and biology, including semilinear parabolic equations on one-dimensional bounded fixed regions, tridiagonal competitive cooperative systems, Cauchy-Riemann equation on $S^1$, and play very important roles in characterizing the global dynamics of such systems (see \cite{H.MATANO:1982,Matano,smillie1984,terescak1994_sub,VMV}). 

The paper is organized as follows. In Section 2, we introduce some basic concept of dynamics. In Section 3, we define an integer-valued Lyapunov function for \eqref{linear system}+\eqref{linear sys assumption}, and develop its properties. In Section 4, we state our main results and give detail proofs. In Section 5, we give a dissipative condition, so that every orbit of \eqref{system 1.1}+\eqref{assume system 1.1} is bounded, and then discuss some relationships between $1$-cooperative and $2$-cooperative systems.

\section{Preliminaries}
In this section, we list some concept for discrete dynamical systems, $k$-cooperative systems, as well as a characterization for nature number set. 

\subsection{Basic concept in dynamical systems}
\begin{definition}\label{defn of dynamic}
{\rm Let X be a topological vector space and G be a topological group. If $\psi:G\times X \rightarrow X$ is continuous and satisfies:
\begin{itemize}
  \item[{\rm (1)}] $\psi(e,x)=x$, where $e$ is an identity element of $G$;
  \item[{\rm (2)}] $g_1,g_2\in G, \psi(g_1,\psi(g_2,x))=\psi(g_1g_2,x)$
\end{itemize}
Then, $(X,G,\psi)$ is called a dynamical system.}
\end{definition}
Obviously, for any  $g\in G$, $\psi(g,\cdot): X\rightarrow X,\ x\mapsto\psi(g,x)$ is 
homeomorphic. If $G=\mathbb{R}$ is an additive group of real numbers, then $(X,\mathbb{R},\psi)$ is called a flow.  If $G=\mathbb{Z}$ is an additive group of integers,  $(X,\mathbb{Z},\psi)$ is called a discrete dynamical system. Let $T:X\rightarrow X$ be a homeomorphism and $\psi:\mathbb{Z}\times X\rightarrow 
X$ be such that $\psi(n,x)=T^{n}(x)$, then $(X,\mathbb{Z},\psi)$ (or $(X,T)$) is a discrete dynamical system as well.

\begin{definition}
{\rm
 Let $(\mathbb{R}^n,T)$ be a discrete dynamical system. Then, the sets $O(x)=\{T^{n}(x)\}_{n\in \mathbb{Z}}$, $
  O^{+}(x)=\{T^{n}(x)\}_{n \in \mathbb{N}}$, $O^{-}(x)=\{T^{-n}(x)\}_{n \in \mathbb{N}}$ are called orbit, positive half orbit and negative half orbit of $x\in\mathbb{R}^n$, respectively.
  }
\end{definition}

\begin{definition}
{\rm
Let $(\mathbb{R}^n,T)$ be a discrete dynamical system. Assume that $ O^{+}(x)$(resp. $ O^{-}(x)$) is bounded, then the $\omega$-limit(resp. $\alpha(x)$-limit ) set of $x$ is defined in the following 
\[
\begin{split}
  \omega^T(x)&=\{z\in \mathbb{R}^{n}\mid  \exists n_{i}\rightarrow +\infty,T^{n_{i}}(x)\rightarrow z \};\\
  (\text{resp.} \alpha^T(x)&=\{z\in \mathbb{R}^{n}\mid  \exists n_{i}\rightarrow +\infty,T^{-n_{i}}(x)\rightarrow z \}.)
\end{split}
\]
Obviously, $\omega^T(x)$, $\alpha^T(x)$ are non-empty, compact, $T$-invariant sets.}
\end{definition}

\begin{definition}\label{insert property}
{\rm
Let $X$ and $\widetilde{X}$ are topological spaces, $(X,T)$ and $(\widetilde{X},\widetilde{T})$ are dynamical systems, $M$ is a compact invariant set in $(X,T)$. 
If there is a homeomorphic mapping $h:M\rightarrow h(M)\subset \widetilde{X}$, such that
\begin{equation*}
  h\circ T=\widetilde{T}\circ h.
\end{equation*}
Then,  $(M,T)$ is called embedded into $(h(\widetilde{M}),\widetilde{T})$.
}
\end{definition}

\begin{definition}\label{transitive}
{\rm
 Let $(X,T)$ be a discrete dynamical system. If there is $x\in X$ such that $\omega^T(x)=X$, then the continuous mapping $T$ on $X$ is called transitive.
 }
\end{definition}

\begin{lemma}\label{inverse-transitive}
 Let $X$ be a $Baire$ space, and $(X,T)$  be a discrete dynamical system, then $T$ is transitive if and only if $T^{-1}$ is transitive.
\end{lemma}

\begin{proof}
  See \cite[Corollary I.11.5]{mane2012_sub}.
\end{proof}

\subsection{$k$-positive and $k$-cooperative systems}\label{k-cooperative}
The concept of $k$-cooperative and $k$-positive systems was recently introduced by E. Weiss, M. Margaliot in \cite{weiss2021_sub}. Consider the following $n$-dimensional linear equation:
\begin{equation}\label{k-positive}
  \dot{x}=A(t)x, x\in \mathbb{R}^n
\end{equation}
where $A(t)$ is continuous in the interval under consideration. Then, \eqref{k-positive} is called {\it $k$-positive} system ($k=1,\cdots,n-1$), if it maps the set of vectors with at most $k-1$ sign variations to itself; moreover, if $A(t)$ is irreducible, it is then a {\it strongly $k$-positive} system. Particularly, if $k=1$, \eqref{k-positive} is then a positive linear system. System \eqref{cooperative} is called a {\it $k$-cooperative} (resp. {\it strongly $k$-cooperative}) system, if and only if that the corresponding linear variation system is a $k$-positive (resp. strongly $k$-positive) system. 

By introducing compound matrices and Metzler matrices,  a $k$-positive system ($k=2,\cdots,n-2$) equivalents to the $k$-th additive compound matrix of $A(t)$ is Metzler (see \cite[Theorem 4]{weiss2021_sub}). In other words, a $k$-positive system means that $A(t)$ of \eqref{k-positive} is of the following form:
\begin{itemize}
  \item  $(-1)^{k-1} a_{1 n}(t),(-1)^{k-1} a_{n 1}(t) \geq 0$;
  \item  $a_{i j}(t) \geq 0$ for all $i, j$ with $|i-j|=1$;
  \item  $a_{i j}(t)=0$ for all $i, j$ with $1<|i-j|<n-1$.
\end{itemize}

\subsection{A property of the natural number set}
\begin{lemma}\label{exist-k}
  Let $\mathcal{M}$ and $\mathcal{N}$ be two infinite subsets of $\mathbb{N}$. Then there are natural numbers $m_1,m_2,m_3
  \in \mathcal{M}$,\ $n_1,n_2,n_3\in \mathcal{N}$ and $k_1,k_2,l_1,l_2$, such that
  \begin{equation*}\label{lem3_41}
    n_3>k_1m_1+k_2m_2+m_1
  \end{equation*}
  and
  \begin{equation*}\label{lem3_42}
    m_3=l_1n_1+l_2l_2+n_3-k_1m_1+k_2m_2.
  \end{equation*}
\end{lemma}

\begin{proof}
 See \cite[Lemma 3.4]{terescak1994_sub}.
\end{proof}

\section{Integer-valued Lyapunov function for a class of strongly 2-positive system}
In this section, we consider a class of strongly 2-positive system, and introduce some properties for an integer-valued Lyapunov function of these system. More precisly, the system we consider here is the following 
\begin{equation}\label{linear system}
  \begin{aligned}
    \dot{x}_1&=a_{1,1}(t)x_1+a_{1,2}(t)x_2+a_{1,n}(t)x_{n},\\
    \dot{x}_i&=a_{i,i-1}(t)x_{i-1}+a_{i,i}(t)x_{i}+a_{i,i+1}(t)x_{i+1},\quad 2\leq i\leq n-1,\\
    \dot{x}_{n}&=a_{n,n-1}(t)x_{n-1}+a_{n,1}(t)x_{1}+a_{n,n}(t)x_{n}.
  \end{aligned}
\end{equation}
in which the coefficient functions with respect to $t$ are continuous over $\mathbb{R}$. Denote $A(t)=(a_{ij}(t))_{n\times n}$ be the coefficient matrix of 
\eqref{linear system}, and assume that
\begin{equation}\label{linear sys assumption}
  \begin{aligned}
    & a_{1,n}\leq 0,  a_{n,1}\leq 0\\
    & a_{i,i-1}(t)\geq 0,\quad 2\leq i\leq n,\\
    & a_{i,i+1}(t)\geq 0,\quad 1\leq i\leq n-1 \\
    & \prod_{i=1}^{n}a_{i,i-1}(t)+\prod_{i=1}^{n}a_{i,i+1}(t)<0.
  \end{aligned}
\end{equation}
Note that $A(t)$ is irreducible, and hence, \eqref{linear system}+\eqref{linear sys assumption} is a strongly $2$-positive system.

Denote $\mathcal{O}=\{x\mid x\in \mathbb{R}^{n}$ and $x_i\neq0, i=1,2,\cdots,n\}$. Then, similar as in \cite{elkhader1992_sub,MP1990_sub}, we introduce the integer-valued Lyapunov function $\sigma(\cdot)$ of system \eqref{linear system}+\eqref{linear sys assumption} in the following
\begin{definition}\label{L-function}
{\rm
For any $x\in \mathcal{O}$, define
\begin{equation*}
  \sigma(x)=card\{i|\delta_{i}x_{i}x_{i-1}\leq 0\}.
\end{equation*}
where $\delta_1=-1, \delta_{i}=1, \  2\leq i \leq n$, and $card$ represents the number of elements contained in the set.
}
\end{definition} 

Define
\begin{equation*}
  \tilde{n}=\left\{
\begin{matrix}
n,\quad \text{if $n$ is odd},\\
n-1,\quad \text{if $n$ is even}.
\end{matrix}
\right.
\end{equation*}
It is obvious that $\mathcal{O}$ is an open and dense set in $\mathbb{R}^{n}$, $\sigma(\cdot)\in\{1,3,5,\cdots,\tilde{n}\}$ and continuous on $\mathcal{O}$. Moreover, 
its domain of definition can be extended(by continuity) to 
$$\Lambda=\{x\mid
x\in \mathbb{R}^{n}, \text{if}\ x_i=0, \text{then}\ \delta_{i}\delta_{i+1}x_{i+1}x_{i-1}<0\}.
$$

\begin{definition}
{\rm
For all $x\in\mathbb{R}^n$, define functions as:
\[
\sigma_m,\ \sigma_M:\mathbb{R}^n\to \{1,3,\cdots,\tilde{n}\}.
\]
where $\sigma_m(x)$ and $\sigma_M(x)$ are the minimum and maximum values of $\sigma (y)$ over $\mathcal{U}\cap \Lambda$, and $\mathcal{U}$ is a sufficiently 
small neighborhood of $x$ in $\mathbb{R}^n$.}
\end{definition}

We have the following characterization about $\sigma(\cdot)$.
\begin{theorem}\label{linear system large-constant}
Let $x(t)$ be a nontrivial bounded solution of system $\eqref{linear system}+\eqref{linear sys assumption}$, then 
\begin{enumerate}
         \item[{\rm(i)}] $x(t)\in\Lambda$, except for finite points of $\mathbb{R}$;
          \item[{\rm(ii)}] $\sigma(x(t))$ is locally constant, where $x(t)\in\Lambda$;
          \item[{\rm(iii)}] if $x(t_{0})\notin\Lambda$, then $\sigma(x(t_ {0}^{+}))<\sigma(x(t_{0}^{-}))$;
          \item[{\rm(iv)}] if $x(t)\in\Lambda$, then $(x_i(t),x_ {i+1}(t))\neq(0,0)$(for each $i\in\{1,2,\cdots,n\}$);
         \item[{\rm(v)}] there exists $t_0>0$ such that when $t\in [t_0,+\infty) \cup (-\infty,-t_0]$, $x(t)\in \Lambda$ and $\sigma(x(t))$ is constant.
\end{enumerate}
\end{theorem}

\begin{corollary}\label{nonlinear system large-constant}
  Let $\phi(t,x)$ and $\phi(t,y)$ are two different bounded solutions of system $\eqref{system 1.1}+\eqref{assume simple system 1.1}$, then
  \begin{enumerate}
         \item[{\rm(i)}]  $\phi(t,x)-\phi(t,y)\in\Lambda$ except for finite points of $\mathbb{R}$;
        \item[{\rm(ii)}]  $\sigma(\phi(t,x)-\phi(t,y))$ is locally constant, where $\phi(t,x)-\phi(t,y)\in\Lambda$;
         \item[{\rm(iii)}]  if $\phi(t_0,x)-\phi(t_0,y)\notin\Lambda$, then $\sigma(\phi(t_0^+,x)-\phi(t_0^+,y))<\sigma(\phi(t_0^-,x)-\phi(t_0^-,y))$ $;$
         \item[{\rm(iv)}]  if $\phi(t,x)-\phi(t,y)\in\Lambda$, then $(\phi(t,x)_{i},\phi(t,x)_{i+1}) \neq (\phi(t,y)_{i},\phi(t,y)_{i+1})$(for each $i\in\{1,2,\cdots,n\}$);
        \item[{\rm(v)}]  there exist $t_0>0 $ and $C_1,C_2\in \{0,1,\cdots,n-1\},$ where $\ C_1\leq C_2 $, such that
	\begin{equation*}\label{const-eq}
		\sigma(\phi(t,x)-\phi(t,y))=C_{1},t\geq t_{0},\quad \sigma(\phi(t,x)-\phi(t,y))=C_{2},t\leq -t_{0}.
	\end{equation*}
         \end{enumerate}
\end{corollary}

\begin{proposition}\label{th-monotone of NmM}
Let $t\rightarrow \tilde A(t)=(\tilde a_{ij}(t))_{n\times n}$ be a continuous matrix function in $(a,b)$, and  $\tilde\varphi(t,x)$ be a nontrivial solution of  
\[\dot x=\tilde A(t)x\]
with an initial value condition 
$\tilde \varphi(t_0,x^0)=x^0$. Then the following two are equivalent:

\begin{enumerate}
 \item[{\rm(i)}] 
there is an open and dense subset $\mathcal{G}$ of $(a,b)$ such that for each $t$ in $\mathcal{G}$, $\tilde A(t)$ is in the form of \eqref{linear sys assumption};
 \item[{\rm(ii)}]
 $\sigma(\tilde\varphi(t,x_0))$ is decreasing with respect to $t\geq t_0$, that is
  \begin{equation*}
  x^0\neq 0, x^0\notin \Lambda\Rightarrow\tilde \varphi(t_0+\epsilon,x^0),\tilde \varphi(t_0-\epsilon,x^0)\in \Lambda,
\end{equation*}
and
\begin{equation*}
  \sigma_m(x^0)=\sigma(\tilde \varphi(t_0+\epsilon,x^0))<\sigma(\tilde \varphi(t_0-\epsilon,x^0))=\sigma_M(x^0)
\end{equation*}
where $\epsilon>0$ is small enough.
\end{enumerate}
\end{proposition}
\par
\begin{remark} 
{\rm
Proposition \ref{th-monotone of NmM} was proved for positive cyclic feedback system in \cite[Theorem 1]{fusco}, that is, $a_{1,n}\geq0$, $a_{n,1}\geq 0$ and $\prod_{i=1}^{n}a_{i,i-1}(t)+\prod_{i=1}^{n}a_{i,i+1}(t)>0$ in \eqref{linear sys assumption}. }
\end{remark}

Based on Proposition \ref{th-monotone of NmM}, we can prove Theorem \ref{linear system large-constant}.

\begin{proof}[ Theorem \ref{linear system large-constant}]
 By the definition of $\sigma(\cdot)$, (i)-(iv) can be drawn directly from Proposition \ref{th-monotone of NmM}.
 By (i) and (iii), $\sigma(x(t))$ only drops strictly a finite number of times with respect to $t$, so (v) is work as well. 
\end{proof}
\vskip 2mm

\begin{proof}[Proof of Corollary \ref{nonlinear system large-constant}]
  Let $z(t)=\phi(t,x)-\phi(t,y)$, then $z(t)$ satisfies the following linear equation:
\begin{equation}\label{differe-equa}
	\dot{z}(t)=\tilde A(t)z(t),
\end{equation}
where $\tilde A=(\tilde {a}_{ij})_{n\times n}$, $\tilde a_{ij}=\int_{0}^{1}\frac{\partial f_i(t,r\phi(t,x)+(1-r)\phi(t,y))}{\partial x_j}dr$. Obviously, 
\eqref{differe-equa} satisfies \eqref{linear system}, then the corollary is true based on theorem \ref{linear system large-constant}.
\end{proof}
\vskip 2mm

In the rest of this section, we prove Proposition \ref{th-monotone of NmM}. Note that if $\sigma_m(x)=\tilde{n}$, then Proposition \ref{th-monotone of NmM} is automatically established. We only consider the case that 
$\sigma_m(x)<\tilde{n}-2$. To prove Proposition \ref{th-monotone of NmM}, some lemmas are listed for preparations. 

For each $h\in\{1,3,\cdots,\tilde{n}-2\}$, 
let $m_h=\{x\in \mathbb{R}^n\backslash\{0\}| \sigma_m(x)=h\},N_h=\partial m_h\cap m_h, \theta_h=\{x|\sigma_m(x)>h\}.$

\begin{lemma}\label{theta openset}
  For each $h\in\{1,3,\cdots,\tilde{n}-2\},\theta_h $is an open set.
\end{lemma}

\begin{proof}
 Suppose on the contrary that $\theta_h$ is not open. Fix $x^0\in \theta_h$ and $1\ll n\in\mathbb{N}$, then there is $x^n\in B_\frac{1}{2n}(x^0)$ ($B_\frac{1}{2n}(x^0)$ represents a open ball with center $x^0$ and radius $\frac{1}{2n}$), such that
  \begin{equation*}
    \sigma_m(x^n)\leq h.
  \end{equation*}
 For the given $x^n$, by the definition of $\sigma_m(\cdot)$, there also exists $y^n\in B_\frac{1}{2n}(x^n)\cap\Lambda$ such that
  \begin{equation*}
    \sigma(y^n)\leq h.
  \end{equation*}
Let $n\rightarrow+\infty,$ then $y^n\rightarrow x^0$, but $\sigma(y^n)\leq h$, which is a contradiction to that $\sigma_m(x^0)> h$. Thus, $\theta_h$ is an open 
set.
\end{proof}
\vskip 2mm

Let $x_{i_1},\cdots,x_{i_k},i_1<\cdots<i_k$, be those non-zero components of $x$, then 
\begin{equation*}
\begin{split}
   x\in m_h\Leftrightarrow\ &\text{the number of sign changes in the sequence}\\ &x_{i_1},\cdots,x_{i_k},-x_{i_1} \text{is } h.
\end{split}
\end{equation*}
Therefore, $x\in N_h$ if and only if $x\in m_h$ and there exists $i$ such that one of the following holds:
\begin{itemize}

\item[\rm (i)] $x_{i-1}=x_{i}=0$;

\item [\rm (ii)] $x_i=0$ and
$\left\{
\begin{aligned}
x_{i-1}x_{i+1}>0, i \neq n,1\\
x_{i-1}x_{i+1}<0, i = n,1.
\end{aligned}
\right.$
\end{itemize}

Moreover, we have the relationship between $\partial m_h$ and $N_h $ as follows.

\begin{lemma}\label{locally coincide}
  $\partial m_h, N_h$ are coincide locally, that is, for each $x\in N_h$, there exists $\delta_x>0$ such that for any $\tilde{x}\in\partial m_h$ with 
  $|\tilde{x}-x|<\delta_x$, one has $\tilde{x}\in N_h$.
\end{lemma}

\begin{proof}
Choose $x\in N_h$ and let $x_{i_1},\cdots,x_{i_k}$($i_1<\cdots<i_k$) be non-zero components of $x$. Then there exists $\delta_0>0$, such that for each 
  $\tilde{x}\in\mathbb{R}^n$ with $|\tilde{x}-x|<\delta_0$ one has $\tilde{x}_ {i_1},\cdots,\tilde{x}_{i_k}$ share the same signs with
  $x_{i_1},\cdots,x_{i_k}$. Thus, $\sigma_m(\tilde{x})\geq h$. 
  
Next, we prove that if $\tilde{x}\in B_{\delta_0}(x)\cap\partial m_h$, then
  \begin{equation*}
   \sigma_m(\tilde{x})= h\Leftrightarrow\tilde{x}\in N_h.
  \end{equation*}
  Indeed, if this is not correct, then 
  \begin{equation*}
      \sigma_m(\tilde{x})> h.
  \end{equation*}
Note that $\theta_h$ is an open set(see Lemma \ref{theta openset}), there exists $\tilde{\delta}>0$, such that 
  \begin{equation*}
    \sigma_m(y)> h,
  \end{equation*}
for any $y\in B_{\tilde{\delta}}(\tilde{x})$, a contradiction to that $\tilde{x}\in\partial m_h$. Thus, $ \sigma_m(\tilde{x})= h$ and $\tilde{x}\in N_h$.
\end{proof}

Let $N_{h,0}=\operatorname{Int} m_h$(interior of $m_h$). For $q=1,\cdots,\tilde{n}-h-2$, define $N_ {h,q}\subset m_h$ be the subset of vectors $x\in m_h$ satisfy
\begin{equation*}
\begin{aligned}
 &x\notin N_ {h,p},\quad p=0,\cdots,q-1;\\
 &x_i=\cdots=x_{i+q}=0\Rightarrow
\left\{
\begin{aligned}
x_{i-1}x_{i+q+1}<0, \quad i+q+1\leq n,\\
x_{i-1}x_{i+q+1-n}>0, \quad  i+q+1>n;
\end{aligned}
\right.
\end{aligned}
\end{equation*}
for $q=\tilde{n}-h-1$, define
\begin{equation*}
  N_{h,\tilde{n}-h-1}= N_h\setminus\bigcup\limits_{p=1}^{\tilde{n}-h-2}N_{h,p}\Leftrightarrow N_h=N_{h,1}\cup\cdots\cup N_{h,\tilde{n}-h-1}.
\end{equation*}
It is not hard to see that if $x\in N_{h,q}$, $x'\in N_h$ with $|x-x'|\ll 1$, then $ x'\in N_{h,p}$ for some $p\leq q$ .

Back to \eqref{linear system}+\eqref{linear sys assumption}, without loss of generality, we may always assume that \eqref{linear sys assumption} is the following special form:
\begin{equation*}
  \begin{aligned}
   a_{1,n}&<0,  a_{n,1}\leq 0\\
    a_{i,i-1}(t)&>0,\quad 2\leq i\leq n,\\
    a_{i,i+1}(t)&\geq 0,\quad 1\leq i\leq n-1 .
  \end{aligned}
\end{equation*}
 
Let $\varphi(t,x)$ be a solution of \eqref{linear system}+\eqref{linear sys assumption}. For each $q\leq \tilde{n}-h-2$, let $x\in N_{h,q}$ be the point such that if $x_ {i-1}\neq 0, x_i=0$,
then one of the following holds:
\begin{itemize}
\item[\rm{(a)}] $i=1, x_nx_2\leq 0$,
$$
x_n\dot{\varphi}(t,x)_1=x_n(a_{1,1}x_1+a_{1,2}x_2+a_{1,n}x_n)=a_{1,n}x_n^2+a_{1,2}x_nx_2<0;
$$
\item[\rm{(b)}]  $i\in \{2,3,\cdots,n-1\}, x_{i-1}x_{i+1}\geq 0$,
$$ 
x_{i-1}\dot{\varphi}(t,x)_i=x_{i-1}(a_{i-1,i}x_{i-1}+a_{i,i}x_i+a_{i,i+1}x_{i+1})>0;
$$
\item[\rm{(c)}]  $i=n, x_{n-1}x_1\leq 0$,
$$%%原来写成a_{n-1,n}x_{n-1}
x_{n-1}\dot{\varphi}(t,x)_n=x_{n-1}(a_{n,1}x_1+a_{n,n-1}x_{n-1}+a_{n,n}x_n)>0.
$$
\end{itemize}

For any $x\in N_h$, let $I_x\subset \{1,\cdots,n\}$ be the set of indices with the above property. Then $I_x$ is non-empty; moreover, we have the following

\begin{lemma}\label{lem Ix cap Ix not nothing}
  Assume that $x, x'\in N_{h,q}$ for some $q\in\{0,\cdots,\tilde{n}-h-1\}$. If $ |x-x'|\ll1$, then
  \begin{equation*}
    I_x\cap I_{x'}\neq\varnothing.
  \end{equation*}
\end{lemma}
\begin{proof}
  Choose $x\in N_{h,q}$ and let $x_{j1},\cdots,x_{jk}$ be non-zero components of $x$. Then for $x'\in N_{h,q}$ with $|x-x'|\ll1$, one has 
 $x'_{j1},\cdots,x'_{jk}$  keep the same signs with $x_{j1},\cdots,x_{jk}$. Suppose that $I_x\cap I_{x'}=\varnothing$ and let $i\in I_x$,  then $\ i\notin I_{x'}, 
  \ x_{i-1}x'_{i-1}>0$. 
  
(a) $i=1$, then $x'_1\neq0$ or $x'_nx'_2>0$. If $x'_nx'_2>0, x_2=0$, then there exists sign change in $x'_n,-x'_2$, which means
  \begin{equation*}
    \sigma_m(x')>h,
  \end{equation*}
 a contradiction to $x'\in N_{h,q}$. If $x'_nx'_2>0, x_2\neq0$, however, this cannot happen due to $x_nx_2<0$. Thus,  $x'_i\neq0$.
 
  Similar arguments for (b): $i\in\{2,3,\cdots,n-1\}$ and (c): $i=n$.
  
As a consequence, for any $i\in I_x $, there is $x'_i\neq0$,  a contradiction to that $x'\notin N_{h,q}$. Thus, $ I_x\cap I_{x'}\neq\varnothing.$
\end{proof}

We now prove Proposition \ref{th-monotone of NmM}.

\begin{proof}[Proof of Proposition \ref{th-monotone of NmM}]
``(i) $\Rightarrow $ (ii)''

Assume that $t_0\in\mathcal{G}$, we first prove that
  $$
  \sigma_m(x_0)=\sigma(\tilde\varphi(t_0+\epsilon,x_0))
  $$
for $0<\epsilon\ll1$. And this equivalents to show that: for each $x^0\in m_h$($h\in\{1,3,\cdots,\tilde n-2\}$), there 
  exists $\epsilon>0$ such that for all $t\in(t_0,t_0+\epsilon)$,
  \begin{equation}\label{intmh}
   \tilde\varphi(t,x)\in \operatorname{Int} m_h\Leftrightarrow\tilde\varphi(t,x)\in\Lambda, \sigma(\tilde\varphi(t,x))=h.
  \end{equation}

Given $q\in \{1,\cdots,\tilde{n}-h-1\}$ and $x\in N_{h,q}$. For any $i\in I_x$, one has
  \begin{equation*}
  \begin{cases}
    x_{i-1}\dot{\tilde\varphi}(t_0,x)_i>0, &  i\in\{2,3,\cdots,n\}, \\
    x_{i-1}\dot{\tilde\varphi}(t_0,x)_i<0, &  i=1.
  \end{cases}
\end{equation*}

Let $0<t-t_0\ll1, \tilde{x}=\tilde\varphi(t,x)$, then $\tilde{x}_i\neq0$ for all $i\in I_x$. Therefore, $I_x\cap I_{\tilde{x}}=\varnothing$. If $q=1$, then 
$\tilde{x}\in \Lambda$, and \eqref{intmh} is established. For $q\in\{2,3,\cdots,\tilde{n}-h-1\}$, suppose that \eqref{intmh} is established for all 
$p<q$. We then prove that \eqref{intmh} is established for $p=q$.
\par
Suppose on the contrary that there exists $x\in N_{h,q}$ with no $\epsilon>0$ satisfies
$$ \tilde\varphi(t,x)\in \operatorname{Int} m_h, t\in(t_0,t_0+\epsilon).$$
Note that 
$$
I_x\cap I_{\tilde\varphi(t,x)}=\varnothing,\, \sigma_m(\tilde\varphi(t,x))\geq h,
$$
for $0<t-t_0\ll1$. By virtue of Lemma \ref{lem Ix cap Ix not nothing}, $\tilde\varphi(t,x)\notin N_{h,q}$. Observe that $\tilde\varphi(t,x)\notin \cup_{i=q+1}^{n-h-1}N_{h,i}$, 

\begin{equation*}
  \sigma_m(\tilde\varphi(t,x))>h, \quad  0<t-t_0\ll1.
\end{equation*}
In other words, there exists $0<\epsilon\ll1$ such that
\begin{equation*}
  \sigma_m(\tilde\varphi(t,x))>h,\quad t\in(t_0,t_0+\epsilon).
\end{equation*}
By Lemma \ref{theta openset} and the continuity of $\tilde\varphi(t,x)$, one can choose $y\in\bigcup\limits_{i=1}^{q-1}N_{h,i}$ with $|y-x|\ll1$, such that
\begin{equation}\label{Nmy>h}
%%原来我写成\frac{\delta}{2}
  \sigma_m(\tilde\varphi(t_0+\frac{\epsilon}{2},y))>h.
\end{equation}
Since $y\in\bigcup\limits_{i=1}^{q-1}N_{h,i}$, there exists $0<\epsilon_y<\frac{\epsilon}{2}$ such that

\begin{equation}\label{phi in intmh}
\tilde\varphi(t,y)\in \operatorname{Int}m_h,\quad t\in(t_0,t_0+\epsilon_y) .
\end{equation}
Combine \eqref{Nmy>h}, \eqref{phi in intmh} and the continuity of $\tilde\varphi$, there are $\bar{t}\in(t_0+\epsilon_y,t_0+\frac{\epsilon}{2}), \delta>0$, such that
\begin{equation}\label{bar t}
  \tilde\varphi(\bar{t},y)\in \partial m_h, \sigma_m(\tilde\varphi(t,y)>h, t\in(\bar{t},\bar{t}+\delta).
\end{equation}
It then follows from Lemma \ref{locally coincide} that 
\begin{equation*}
  \tilde\varphi(\bar{t},y)\in N_h.
\end{equation*}
Thus, \eqref{bar t} implies that $\tilde\varphi(\bar{t},y)\in N_{h,q} $.

We now further restrict $y$ to satisfies:
\begin{equation*}
\begin{aligned}
  (i)\quad x_{i-1}y_i>0,\quad &\text{ if }\ i\in I_x\cap\{2,3,\cdots,n\}, \\
  (ii)\quad x_{i-1}y_i<0,\quad &\text{ if }\ i\in I_x\cap\{1\}.
\end{aligned}
\end{equation*}
Recall that $I_{\phi(\bar{t},y)}\cap I_x\neq \varnothing$, {\it we claim that: there exist $i_0\in I_x, \bar{\bar{t}}\in (t_0,\bar{t}]$ such that

\begin{equation*}
\begin{aligned}
x_{i_0-1}\dot{\tilde\varphi}(\bar{\bar{t}},y)_{i_0}\leq0,\quad &\text{ if }\ i_0\neq1,\\
x_{i_0-1}\dot{\tilde\varphi}(\bar{\bar{t}},y)_{i_0}\geq0,\quad &\text{ if }\ i_0=1. 
 \end{aligned}
\end{equation*}}
Indeed, if the claim is not hold, then for any $i\in I_x, t\in  (t_0,\bar{t}]$, one has
\begin{equation*}
  \begin{cases}
    x_{i-1}\dot{\tilde\varphi}(t,y)_i>0, &  i\in\{2,3,\cdots,n\}, \\
    x_{i-1}\dot{\tilde\varphi}(t,y)_i<0, &  i=1.
  \end{cases}
\end{equation*}
And hence, $\tilde\varphi(t,y)_i$ is monotone with respect to $t$, that is, $\phi(t,y)_i\neq0$. Thus,  $I_x\cap I_{\phi(\bar t,y)}=\varnothing$, which is a contradiction.
The claim is proved.

By continuity of $\tilde\varphi$, for $i\in I_x$, one has:
\begin{equation*}
\begin{aligned}
(i)\quad x_{i-1}\dot{\tilde\varphi}(t,x)_i>0,\quad &\text{ if }\ i\in\{2,3,\cdots,n\},0<t-t_0\ll1,\\
(ii)\quad x_{i-1}\dot{\tilde\varphi}(t,x)_i<0,\quad &\text{ if }\ i=1,0<t-t_0\ll1. 
 \end{aligned}
\end{equation*}
By continuity of $\dot{\tilde\varphi}$, we can further get $i\in I_x$:
\begin{equation*}
\begin{aligned}
(i)\quad x_{i-1}\dot{\tilde\varphi}(t,y)_i>0,\quad &\text{ if } i\in\{2,3,\cdots,n\},0<t-t_0\ll1,\\
(ii)\quad x_{i-1}\dot{\tilde\varphi}(t,y)_i<0,\quad &\text{ if }\ i=1,0<t-t_0\ll1. 
 \end{aligned}
\end{equation*}
Choose $t=\bar{\bar{t}}$, then it is a contradiction to the claim. 

Consequently, \eqref{intmh} is also correct for $x\in N_{h,q}, t\in \mathcal{G}$.

Let $M_h=\{x\in \mathbb{R}^n\backslash\{0\}| \sigma_M(x)=h\},\tilde{N}_h=\partial M_h\cap M_h$. By appropriate modification of the above proof, one can obtain that for each $x^0\in M_h$, there 
  exists $\epsilon_0>0$ such that for all $t\in(t_0-\epsilon_0, t_0)$,
  \begin{equation}\label{intMh}
    \tilde\varphi(t,x)\in \operatorname{Int} M_h\Leftrightarrow\tilde\varphi(t,x)\in\Lambda, \sigma(\tilde\varphi(t,x))=h.
  \end{equation}
  
To complete the proof of (i) $\Rightarrow$ (ii), it is sufficient to check that \eqref{intmh} and \eqref{intMh} also hold for some $\epsilon_0>0$, when $t_0\in (a,b)\setminus \mathcal{G}$. For this purpose, fix ${\rd \epsilon_0}>0$ and let $\tilde A^k(t)$ be a sequence of continuous functions in the form of \eqref{linear sys assumption}, such that $\tilde A^k(t)\rightarrow \tilde A(t)$ uniformly in $[t_0-\epsilon_0, t_0+\epsilon_0]$, as $k\rightarrow \infty$. Let $\tilde\varphi^k(t,x)$ be the solution of $\dot{x}=\tilde A^k(t)x$ through $x$ at time $t_0$. 

By the Gronwall inequality, $\tilde\varphi^k(t,x)\rightarrow \tilde\varphi(t,x)$ as $k\rightarrow \infty$, uniformly in $[t_0-\epsilon_0, t_0+\epsilon_0]$. As we have proven that: for $\dot{x}=\tilde A^k(t)x$, if $x\neq0, x\notin \Lambda, t_1 \in [t_0-\epsilon_0,t_0]$,$t_2\in [t_0, t_0+\epsilon_0]$, then 

\begin{equation}\label{sigmam phik leq sigmaM phik}
  \sigma_m(\tilde\varphi^k(t_2,x))\leq \sigma_m(x)<\sigma_M(x)\leq \sigma_M(\tilde\varphi^k(t_1,x)).
\end{equation}
Note that $\{y \in \mathbb{R}^n|\sigma_m(x^{'})\leq \sigma_m(x)\}$ and $\{y \in \mathbb{R}^n|\sigma_M(x^{'})\geq \sigma_M(x)\}$  are closed sets,
\begin{equation}\label{sigmam phi leq sigmaM phi}
  \sigma_m(\tilde\varphi(t_2,x))\leq \sigma_m(x)<\sigma_M(x)\leq \sigma_M(\tilde\varphi(t_1,x)).
\end{equation}
{\it We claim that: one can choose $\epsilon_0$ small enough, such that
\begin{equation}\label{simple-1}
  \tilde\varphi(t_0-\epsilon,x),\tilde\varphi(t_0+\epsilon,x)\in \Lambda,
\end{equation}
and
\begin{equation}\label{simple-2}
  \sigma_m(x)=\sigma(\tilde\varphi(t_0+\epsilon,x))<\sigma(\tilde\varphi(t_0-\epsilon,x))=\sigma_M(x),
\end{equation}
for all $0< \epsilon < \epsilon_0$}. Suppose on the contrary that, there exists $\epsilon_n\to 0$ such that
\[
\tilde\varphi(t_0-\epsilon_n,x)\notin \Lambda.
\]
For simplicity, we assume that $\epsilon_n \downarrow 0$, and note that $\mathcal{G}$ is an open and dense subset of $(a,b)$, there exist $t_n\in(t_0-\epsilon_{n-1},t_0-\epsilon_{n})$ such that $\tilde \varphi(t_n,x)\in \Lambda$; moreover, by \eqref{sigmam phi leq sigmaM phi} 
\[
\begin{split}
\sigma(\tilde\varphi(t_{n+1},x))&=\sigma_m(\tilde\varphi(t_{n+1},x))\leq \sigma_m(\tilde\varphi(t_0-\epsilon_{n},x))\\
&<\sigma_M(\tilde\varphi(t_0-\epsilon_{n},x))\leq \sigma_M(\tilde\varphi(t_{n},x))\\
&=\sigma(\tilde\varphi(t_{n},x))
\end{split}
\]
which is a contradiction to that $\sigma(\cdot)$ is a finite valued function. Thus, one can choose $\epsilon_0$ small enough such that 
\[
\tilde \varphi(t_0-\epsilon,x)\in\Lambda,\quad 0<\epsilon<\epsilon_0.
\]
Similarly arguments can also lead to
\[
\tilde \varphi(t_0+\epsilon,x)\in\Lambda,\quad 0<\epsilon<\epsilon_0.
\]
It is not hard to see that there is $h\in\{1,3\cdots,\tilde n\}$ such that
\[
\sigma(\tilde\varphi(t_0+\epsilon,x))=h,\quad  0<\epsilon<\epsilon_0.
\]
By \eqref{sigmam phi leq sigmaM phi}, $h\leq \sigma_m(x)$. Note that $\tilde\varphi(t_0+\epsilon,x)\to x$, as $\epsilon\to 0$, by the definition of $\sigma_m(\cdot)$, $h\geq \sigma_m(x)$. As a consequence,
\[
\sigma(\tilde\varphi(t_0+\epsilon,x))=\sigma_m(x),\quad  0<\epsilon<\epsilon_0.
\]
Similarly, one can also get
\[
\sigma(\tilde\varphi(t_0-\epsilon,x))=\sigma_M(x),\quad  0<\epsilon<\epsilon_0.
\]
We have checked  \eqref{simple-2}, and the claim is proved.
\vskip 3mm

``(ii) $\Rightarrow$ (i)''.

Choose $x^0\in \mathbb{R}^n$ be such that $x^0_i=0$, $x^0_j\neq 0$, $j\neq i$ and
\begin{equation*}
\begin{cases}
    x^0_{i-1}x^0_{i+1}>0, \text{ if } 2\leq i\leq n-1 \\
   x^0_{i-1}x^0_{i+1}>0, \text{ if } i=1,n\quad (x^0_0=x^0_n, x^0_{n+1}=x^0_1)
  \end{cases}   
\end{equation*}

Then, $x^0\notin \Lambda$ and $\sigma_m(x^0)=\sigma_M(x^0)-2$.

Let $y\in \mathbb{R}^n$ be with $|y-x^0|\ll1$, then $y\in \Lambda$, if and only if $y_i\neq0$, and one of the following holds true:
\begin{enumerate}
    \item[{\rm(a)}] if $2\leq i \leq n,\quad \left\{
\begin{aligned}
\sigma(y)=\sigma_m( x^0)\Leftrightarrow x^0_{i-1}y_i>0,\\
\sigma(y)=\sigma_M(x^0)\Leftrightarrow x^0_{i-1}y_i<0;
\end{aligned}
\right.$
    \item[{\rm(b)}] if $i=1,\quad \left\{
\begin{aligned}
\sigma(y)=\sigma_m( x^0)\Leftrightarrow x^0_{i-1}y_i<0,\\
\sigma(y)=\sigma_M( x^0)\Leftrightarrow x^0_{i-1}y_i>0.
\end{aligned}
\right.$
\end{enumerate}
Assume that $\tilde\varphi (t_0,x)=x$, then
\begin{equation*}
\begin{cases}
 (\tilde A(t_0)x^0)_ix^0_{i-1}\geq 0,\quad 2\leq i\leq n,\\
 (\tilde A(t_0)x^0)_ix^0_{i-1}\leq 0,\quad i=1.
\end{cases}
\end{equation*}
That is,
\begin{equation*}
\begin{cases}
 x^0_{i-1}(\tilde a_{i,i-1}(t_0)x^0_{i-1}+\tilde a_{i,i+1}(t_0)x^0_{i+1}+\sum\limits_{j\notin \mathcal{I}_i}\tilde a_{ij}(t_0)x^0_j)\geq 0,
    i\in \{2,\cdots,n\}\\
 x^0_{i-1}(\tilde a_{i,i-1}(t_0)x^0_{i-1}+\tilde a_{i,i+1}(t_0)x^0_{i+1}+\sum\limits_{j\notin \mathcal{I}_i}\tilde a_{ij}(t_0)x^0_j)\leq 0,
    i=1.
\end{cases}
\end{equation*}
(where $\mathcal{I}_i=\{i-1,i+1\}$), which implies that for all $ t\in (a,b)$,
 \par
\begin{equation*}
\left\{
\begin{aligned}
    &\tilde a_{ij}(t)=0,\quad \text{for}\ 1<|i-j|<n-1, \\
    &\tilde a_{i,i+1}(t)\geq0,\quad \text{for}\ i=1,\cdots,n-1\\
    &\tilde a_{i-1,i}(t)\geq0,\quad \text{for}\ i=2,\cdots,n\\
    &\tilde a_{1,n}(t),\tilde a_{n,1}(t)\leq0
\end{aligned}
\right.
\end{equation*}
\vskip 3mm

Let $\mathcal{G}$ be the subset of $(a,b)$ such that
\begin{equation*}
  \prod_{i=1}^{n}\tilde a_{i,i-1}(t)+\prod_{i=1}^{n}\tilde a_{i,i+1}(t)<0.
\end{equation*}
Then, by continuity of $\tilde A(t)$, $\mathcal{G}$ is open. Suppose that $\mathcal{G}$ is not dense, then there is an interval $(\overline{a},\overline{b})\subset(a,b)$ and indices $h,k$ 
such that 
$$
\tilde a_{h,h-1}(t)=\tilde a_{k,k+1}(t)=0,\quad t\in(\overline{a},\overline{b}).
$$ 
Define $I_{hk}$ in the following:
\begin{equation*}
  \mathcal{I}_{hk}=\left\{
  \begin{aligned}
    &\{i|h\leq i\leq k\},\quad \text{if}\ h\leq k,\\
    &\{i|i>h\ or\ i \leq k\},\quad \text{if}\ h>k.
  \end{aligned}
  \right.
\end{equation*}
Choose $x\in\mathbb{R}^n\setminus \{0\}$ be such that $x_j=0$, for all $j\in \mathcal{I}_{hk}$. Then,
\begin{equation*}
  \tilde\varphi(t,x)_j=0,\quad j\in \mathcal{I}_{hk},\ t\in(\overline{a},\overline{b}),
\end{equation*}
a contradiction to (ii). This proves (ii) $\Rightarrow$ (i).

The proof of Proposition \ref{th-monotone of NmM} is completed.

\end{proof}
\par

\section{Embedding property of limit sets}
As an application of $\sigma(\cdot)$, we here give a characterization for the asymptotic behaviors of bounded solutions of \eqref{system 1.1}+\eqref{assume simple system 1.1}.

Let $\phi(t,x^0)$ be a solution of \eqref{system 1.1}+\eqref{assume simple system 1.1} with an initial value $x^0$. Denote $P(x^0)=\phi(T,x^0)$ be the associated Poincar\'{e} map. Assume $O^{+}(x^0)=\{y|P^n(x^0),n\in \mathbb{N}\}$ is a bounded set in $\mathbb{R}^n$. Then, our main results in this section are the following

\begin{theorem}\label{changzhixing-th}
  For any $x,y \in \omega^{p}(x^{0}) \ (x\neq y)$ and $t\in \mathbb{R}$, one has
\begin{equation*}\label{eq-changzhixing}
\phi(t,x)-\phi(t,y)\in\Lambda \ \text{and }\sigma(\phi(t,x)-\phi(t,y))\equiv constant.
  \end{equation*}
\end{theorem}
\begin{theorem}\label{main}
  There is a homeomorphic mapping $$h:\omega^p(x^0)\rightarrow h(\omega^p(x^0)) 
  \subset \mathbb{R}^2.$$
\end{theorem}

\subsection{Proof of Theorem \ref{changzhixing-th}}

Similar to the deduction for \cite[Theorem 
A]{terescak1994_sub}, the proof of Theorem \ref{changzhixing-th} is divided into the following Propositions \ref{prop1}, \ref{prop2} and \ref{prop3}. Note that $\sigma$ is defined only on $\Lambda$, the result in \cite[Theorem A]{terescak1994_sub} can not be applied to Theorem 
\ref{changzhixing-th} directly.

\begin{proposition}\label{prop1}
Assume $\widetilde{x}\in \omega^p(x^0)$, then  
  \begin{equation*}
  \phi(t,x)-\phi(t,y)\in\Lambda, \mbox{and} \  \sigma(\phi(t,x)-\phi(t,y))\equiv constant,
  \end{equation*} 
  for all $ x,y\in\overline{O(\widetilde{x})}( x\neq y) \mbox{and} \ t\in \mathbb{R}$.
\end{proposition}

\begin{proposition}\label{prop2}
  Assume $x^0\notin \omega^{p}(x^0),\ x,y\in\omega^{p}(x^0)(x\neq y).$ If $x\notin\alpha^ {p}(x)$ or $y\notin\alpha^{p}(y)$, then
  \begin{equation*}\label{e3}
    \sigma(\phi(t,x)-\phi(t,y))\equiv constant,
  \end{equation*}
  for all $t\in \mathbb{R}$.
\end{proposition}

\begin{proposition}\label{prop3}
  Assume $x^0\notin\omega^{p}(x^0),\ x,y\in\omega^{p}(x^0)(x\neq y).$ If $x\in\alpha^ {p}(x)$ and $y\in\alpha^{p}(y)$, then
\begin{equation*}
\sigma(\phi(t,x)-\phi(t,y))\equiv constant,
\end{equation*}
for all $t\in \mathbb{R}$.
  \end{proposition}
  
\begin{proof}[Proof of Theorem \ref{changzhixing-th}]
  Choose $x^0\in \mathbb{R}^n$  such that $\overline{O^{+}(x^0)} $ is bounded. If $x^0\in\omega^{p}(x^0)$, then 
  the theorem is proved by Proposition \ref{prop1}; if $x^0\notin\omega^{p}(x^0)$, then the theorem is proved by Propositions \ref{prop2} and \ref{prop3}.
\end{proof}

In the following, we will give detailed proofs for Propositions \ref{prop1}, \ref{prop2}, \ref{prop3}.

\begin{proof}[Proof of Proposition \ref{prop1}]
For any $x,y\in \omega^{p}(x^0)(x\neq y)$, by virtue of Corollary \ref{nonlinear system large-constant}(v),  
there exist $C_1,C_2\in\mathbb{N}$, $t_0>0$ such that
  \begin{equation*}
  \begin{aligned}
    \sigma(\phi(t,x)-\phi(t,y))&=C_1,\quad t\geq t_0,\\
    \sigma(\phi(t,x)-\phi(t,y))&=C_2,\quad t\leq - t_0.
  \end{aligned}
  \end{equation*}
Particularly, take $t=n_0T>t_0$, then
\begin{equation*}
    \phi(n_0T,x)-\phi(n_0T,y),\quad \phi(-n_0T,x)-\phi(-n_0T,y)\in \Lambda.
\end{equation*}

The rest of the proof can be divided into the following two cases.

(1) $ x,y\in O(\widetilde{x}),x\neq y$. At this point, there exist $m,n\in \mathbb{Z}$\ such that $ 
\phi(mT,\widetilde{x})=x,\phi(nT,\widetilde{x})=y$, thus,
\begin{equation*}
  \begin{aligned}
    \phi(t,x)-\phi(t,y)=& \phi(t,\phi(mT,\widetilde{x}))-\phi(t,\phi(nT,\widetilde{x}))\\
=&\phi(t+mT,\widetilde{x})-\phi(t+nT,\widetilde{x}).
  \end{aligned}
\end{equation*}
Let $I(t)=\phi(t+mT,\widetilde{x})-\phi(t+nT,\widetilde{x})$. Noticing that $\widetilde{x}\in \omega^p(x^0)$, there exists $\{l_k\}\subset\mathbb{N}$, such that $ 
\phi(l_kT,x^0) \rightarrow \widetilde{x}$, that is,
\begin{equation*}
  \begin{aligned}
    \|I(n_0T)-[\phi((n_0+l_k)T,\phi(mT,x^0))-\phi((n_0+l_k)T,\phi(nT,x^0)]\|\ll 1, \quad k\gg 1
  \end{aligned}
\end{equation*}
where $\|\cdot\|$ represents the 2-norm on $\mathbb{R}^2$. Recall that $\Lambda$ is an open set, then
    \begin{equation}\label{positive-cons}
      \begin{aligned}
       & \phi((n_0+l_k)T,\phi(mT,x^{0}))-\phi((n_0+l_k)T,\phi(nT,x^{0}))\in \Lambda,\quad  k\gg 1,\\
      &\sigma(\phi((n_0+l_k)T,\phi(mT,x^0))-\phi((n_0+l_k)T,\phi(nT,x^0))=C_1,\quad k\gg 1.
     \end{aligned}
    \end{equation}
Similarly, 
   \begin{equation}\label{negative-cons}
      \begin{aligned}
    & \phi((-n_0+l_k)T,\phi(mT,x^{0}))-\phi((-n_0+l_k)T,\phi(nT,x^{0})\in \Lambda, k\gg 1,\\
    &\sigma(\phi((-n_0+l_k)T,\phi(mT,x^0))-\phi((-n_0+l_k)T,\phi(nT,x^0))=C_2, k\gg 1.
    \end{aligned}
   \end{equation}
   %%原来是by (iv)
  By Corollary \ref{nonlinear system large-constant}(v), there exists $C\in\mathbb{N}$ such that
    \begin{equation*}
      \begin{aligned}
       & \phi(t+l_{k}T,\phi(mT,x^{0}))-\phi(t+l_{k}T,\phi(nT,x^{0}))\in\Lambda,\quad  \\ 
       & \sigma(\phi(t+l_{k}T,\phi(mT,x^{0}))-\phi(t+l_{k}T,\phi(nT,x^{0})))=C,
       \end{aligned}
    \end{equation*}
for $k\gg 1$. Combining \eqref{positive-cons} and \eqref{negative-cons}, $C_1=C_2=C$, and by Corollary \ref{nonlinear system large-constant}(iii)
\begin{equation*}
  \phi(t,x)-\phi(t,y)\in\Lambda, \ \text{for all } t \in \mathbb{R}.
\end{equation*}

(2) $x,y\in \overline{O(\widetilde{x})}$. Note that there exist $\{m_s\},\{n_s\}\subset \mathbb{Z}$ such that $\phi(m_ sT,\widetilde{x})\rightarrow x,\ 
\phi(n_sT,\widetilde{x})
  \rightarrow y$,then exists $s_0>0$ such that
  \begin{equation*}
    \begin{aligned}
        \sigma(\phi(n_0T,\phi(m_{s_0}T,\widetilde{x}))-
        \phi(n_0T,\phi(n_{s_0}T,\widetilde{x})))=C_1,\\
        \sigma(\phi(-n_0T,\phi(m_{s_0}T,\widetilde{x}))-
        \phi(-n_0T,\phi(n_{s_0}T,\widetilde{x})))=C_2.
    \end{aligned}
  \end{equation*}
Let $x^{*}=\phi(m_{s_0}T,\widetilde{x}),\ y^{*}=
  \phi(n_{s_0}T,\widetilde{x}).$ According to (1), we know that $\sigma(\phi(t,x^{*})
  -\phi(t,y^{*}))=C $,\ therefore,
  \begin{equation*}
    C_1=C_2=C.
  \end{equation*}
Again by Corollary \ref{nonlinear system large-constant}(iii)
\begin{equation*}
  \phi(t,x)-\phi(t,y)\in\Lambda, \ \text{for all } t \in \mathbb{R}.
\end{equation*}

The proof of Proposition \ref{prop1} is completed.
\end{proof} 
\vskip 2mm

\begin{proof}[Proof of Proposition \ref{prop2}]
Assume that there exist $x^{'},y^{'} \in \omega^{p}(x^0)(x^{'}\neq y^{'})$ but $x^ {'}\notin\alpha^{p}(x^{'})$
  (the proof of $y'\notin\alpha^{p}(y')$ is similar), and exists $t_1\in\mathbb{R}$ such that $\phi(t_1,x)-\phi(t_1,y)\notin \Lambda$. Without loss of 
  generality, take $t_1=0$. Based on Corollary \ref{nonlinear system large-constant}(v), let $n_0\in \mathbb{N} $ be such that $n_0T>t_0$, then
\begin{equation*}
    \phi(n_0T,x^{'})-\phi(n_0T,y^{'}),\quad \phi(-n_0T,x^{'})-\phi(-n_0T,y^{'})\in \Lambda,
\end{equation*}
 and 
  $$
  \sigma(\phi(-n_0T,x^{'})-\phi(-n_0T,y^{'}))> \sigma(\phi(n_0T,x^{'})-\phi(n_0T,y^{'})).
  $$
Let $x=\phi(-n_0T,x^{'}),y=\phi(-n_0T,y^{'})$, we have
  \begin{equation}\label{proof of prop2 no.1}
  \begin{split}
     x-y,\ \phi(2n_0T,x)-\phi(2n_0T,y)\in\Lambda, \\
    \sigma(x-y)>\sigma(\phi(2n_0T,x)-\phi(2n_0T,y)).
  \end{split}
  \end{equation}
Moreover, $x\in O(x^{'})$, and thus, $x\notin \alpha^{p}(x)$.
\par 
Due to $x-y\in \Lambda$, and $\alpha^{p}(x)\subset\overline{O(x)}, \ \alpha^{p}(x)\times\{x\}\subset \overline{O(x)}\times\overline{O(x)}$ are compact set, there are open sets
  $U,V,U_1,\cdots,U_k$ with $x\in U,\ y\in V,\ \alpha^p(x)\subset \bigcup_{i=1}^{k}U_{i}$, such that 
\begin{equation}\label{neiborhood-constant}
  \overline{x}-\overline{y}\in \Lambda\quad \text{for all } (\overline{x},\overline{y})\in U_{i}\times U\, \text{or } U\times V.
\end{equation}
Note that $\phi(2n_0T,x)-\phi(2n_0T,y)\in \Lambda$, choose $U,V$ smaller if necessary, then
  \begin{equation}\label{proof of prop2 no.2}
    \begin{aligned}
   \sigma(\phi(2n_0T,x)-\phi(2n_0T,y))=\sigma(\phi(2n_0T,U)-\phi(2n_0T,V)),\
    \end{aligned}
  \end{equation}
where $\phi(2n_0T,U)-\phi(2n_0T,V)$ denotes the subtraction of elements in sets $\phi(2n_0T,U)$ and $\phi(2n_0T,V)$.

\par
Recall that $x\in\omega^p(x^0)$, then $U\cap O^{+}(x^0)\neq\varnothing$.
Take $x^*\in U \cap O^{+}(x^0)$, it then follows from $\alpha^p(x)\subset\bigcup_{i=1}^{k}U_{i},
  \ y\in\omega^p(x^0)=\omega^p(x^*)$, that there exists $ n_1> n_0$ such that 
  $$ 
  \phi(n_1T,x^*)\in V , \phi(-n_1T,x)\in U_{i} ,\ 1\leq i\leq k.
  $$
Therefore,
  \begin{equation*}
        \sigma(\phi(-n_1T,x)-x)=\sigma(\phi(-n_1T,x)-x^*),\ 
        \sigma(x-\phi(n_1T,x^*))=\sigma(x-y).
  \end{equation*}
  Combining Corollary \ref{nonlinear system large-constant}(ii),(iii), 
  \begin{equation}\label{proor of prop2 no.3}
    \sigma(\phi(-n_1T,x)-x)\geq\sigma(x-y).
  \end{equation}
Note also that $(x^*,\phi(n_1T,x^*))\in U\times V$,
  \begin{equation*}  
      \sigma(\phi(2n_0T,x^*)-\phi((2n_0+n_1)T,x^*))=\sigma(\phi(2n_0T,U)-\phi(2n_0T,V)).
  \end{equation*}
Combining \eqref{neiborhood-constant}-\eqref{proof of prop2 no.2},
  \begin{equation*}
       \sigma(\phi(2n_0T,x)-\phi(2n_0T,y)) =\sigma(\phi(2n_0T,x^*)-\phi((2n_0+n_1)T,x^*)).
  \end{equation*}
then according to \eqref{proof of prop2 no.1} we have
\begin{equation}\label{proof of prop2 no.4}  
      \sigma(x-y)>\sigma(\phi(2n_0T,x^*)-\phi((2n_0+n_1)T,x^*)).
  \end{equation}
Note that $x\in \omega^p(x^*)$, there exists a sequence $\{l_k\} \subset\mathbb{N}$, such that  $\phi(l_kT,x^{*}) \rightarrow x$
and
  \begin{equation*}
        \phi(l_kT,x^*)-\phi((n_1+l_k)T,x^*)\in \Lambda.
  \end{equation*}
That is, there also exists $L>2n_0+n_1$, such that
  \begin{equation*}
        \sigma(\phi(LT,x^*)-\phi((n_1+L)T,x^*))=\sigma(x-\phi(n_1T,x)).
  \end{equation*}
It then follows by Corollary \ref{nonlinear system large-constant}(ii),(iii) and \eqref{proof of prop2 no.4} that
\begin{equation*}
        \sigma(x-y)>\sigma(x-\phi(n_{1}T,x)).
  \end{equation*}
Furthermore, by \eqref{proor of prop2 no.3},
  \begin{equation*}
    \sigma(\phi(-n_1T,x)-x)>\sigma(x-\phi(n_1T,x)),
  \end{equation*}
which contradicts to the case $\widetilde{x}=x$ in Proposition \ref{prop1}. The proof is completed.

\end{proof}

To show Proposition \ref{prop3}, we give the following lemma for preparation.

\begin{lemma}\label{large-x,y}
Let $x,y\in \omega^p(x^0),x\in \omega^p(x),y\in \omega^p(y), x\notin \omega^p(y),y\notin \omega^p(x)$ and $x-y\in \Lambda$, then there exist open sets $U,V$ of 
$x,y$ and a natural number $N$, such that
 \begin{align}
 \sigma(\phi (0^+,x^*)-\phi (0^+,y^*))&\geq\sigma(\phi((mT)^+,x)-\phi((mT)^+,y^*)),\label{large-x}\\
 \sigma(\phi (0^+,x^{'})-\phi (0^+,y^{'}))&\geq\sigma(\phi((mT)^+,x^{'})-\phi((mT)^+,y))\label{large-y},\forall m\geq N,
 \end{align}
for all  $(x^*,y^*)\in U \times\omega^{p}(y),(x^{'},y^{'})\in\omega^{p}(x)\times V$.
Meanwhile, one has 
$$
\overline{x}-\overline{y}\in \Lambda,\quad \sigma(\overline{x}-\overline{y})\equiv constant,
$$
for all $(\overline{x},\overline{y})\in U\times V$.
\end{lemma}

\begin{proof}
  We only prove that \eqref{large-x} is correct, while the proof of \eqref{large-y} is similar.
   \par
  Choose $y^{*}\in\omega^{p}(y)$, then $x\notin\omega^{p}(y)$. According to Corollary \ref{nonlinear system large-constant}(v), one can choose $n_{y^*}=n_0(x,y^*)\in \mathbb{N}$ with $n_{y^*}T>t_0$, such that $\phi(n_ {y^*}T,x)-\phi(n_{y^*}T,y^*)\in \Lambda$. Thus, there exist $U_{y^*},V_{y^*}$ such that
  \begin{equation*}
    \sigma(\phi(n_{y^*}T,x)-\phi(n_{y^*}T,y^*))=\sigma(\phi(n_{y^*}T,U_ {y^*})-\phi(n_{y^*}T,V_{y^*})),
  \end{equation*}
  where $\phi(n_{y^*}T,U_ {y^*})-\phi(n_{y^*}T,V_{y^*})$ represents the subtraction of elements in between $\phi(n_{y^*}T,U_ {y^*})$ and $\phi(n_{y^*}T,V_ 
  {y^*})$.
  By virtue of the compactness of $\omega^p(y)$, there exists a finite subset $\{y^1,y^2,\cdots,y^l\}$  of $\omega^{p}(y) $ such that $V_{y^{i}}(i=1,2,\cdots,l)$ is an open cover of
  $\omega^{p}(y)$. Let: 
  \begin{equation*}
     U=\bigcap\limits_{i=1}^{l} U_{y^{i}},\quad N=\max_{1\leq i \leq l} n_{y^{i}}.
  \end{equation*}
Then, one has  $y^*\in V_{y^{i}},i\in\{1,2,\cdots,l\}$ for any $(x^*,y^*)\in U \times\omega(y)$. Therefore,
 \begin{equation*}
  \begin{aligned}
 \sigma(\phi (0^+,x^*)-\phi(0^+,y^*)) &\geq\sigma(\phi((n_{y^{i}}T),x^*)-\phi((n_{y^{i}}T),y^*))\\
   &=\sigma(\phi (n_{y^{i}}T),x)-\phi((n_{y^{i}}T),y^*))\\
   &\geq \sigma(\phi((mT)^+,x)-\phi((mT)^+,y^*)).
  \end{aligned}
  \end{equation*}
  \par
Similarly, we can find $V$ and probably different $N\in\mathbb{N}$, which satisfy \eqref{large-y} on $\omega^p(x)\times V$. By Corollary \ref{nonlinear 
system large-constant}(ii),(iii), one can find  $N$ large enough such that both \eqref{large-x} and \eqref{large-y} hold together. 

Observe that  $x-y\in\Lambda$, if necessary, smaller $U , V$ can 
be taken such that $\sigma(\overline{x}-\overline{y})=constant$ for all $(\overline{x},\overline{y})\in U \times V$.
\end{proof}

We now prove Proposition \ref{prop3}.

\begin{proof}[Proof of Proposition \ref{prop3}]
Assume that there exist $x^{'},y^{'} \in \omega^{p}(x^{0}) \ (x^{'}\neq y^{'})$ satisfy 
$x'\in\alpha^{p}(x'),y'\in\alpha^{p}(y')$, and there is $t_1\in\mathbb{R}$, such that $\phi(t_1,x')-\phi(t_1,y')\notin \Lambda$. Without loss of generality, we set $t_1=0$. By Corollary \ref{nonlinear system large-constant}(v), one can choose $n_0\in \mathbb{N}$ with $n_0T>t_0$, such that
\begin{equation*}
\begin{aligned}
    \phi(n_0T,x^{'})-\phi(n_0T,y^{'}),\quad \phi(-n_0T,x^{'})-\phi(-n_0T,y^{'})\in \Lambda,\\
    \sigma(\phi(-n_{0}T,x^{'})-\phi(-n_{0}T,y^{'}))>\sigma(\phi(n_{0}T,x^{'})-\phi(n_{0}T,y^{'})).
\end{aligned}
\end{equation*}
If $x^{'}\in\alpha^{p}(y^{'})(similar \ to \ y^{'}\in\alpha^{p}(x^{'})$), then $x^{'}\in \overline{O(y^{'})}$, which is a contradiction to Proposition \ref{prop1}. Thus, 
we may assume that 
$x^{'}\notin\alpha^{p}(y^{'}),y^{'}\notin\alpha^{p}(x^{'})$. Let $x^*=\phi(-n_0T,x^{'}),y^*=\phi(-n_0T,y^{'})$, we have
\begin{equation*}\label{proor of prop3 no.1}
\sigma(x^*-y^*)>\sigma(\phi(2n_0T,x^*)-\phi(2n_0T,y^*)),
\end{equation*}
Apparently, 
$$
\begin{aligned}
& x^*-y^*,\phi(2n_0T,x^*)-\phi(2n_0T,y^*)\in \Lambda, \\
& x^*\notin\alpha^{p}(y^*),y^*\notin\alpha^{p}(x^*).
\end{aligned}
$$

Let $X=\alpha^p(x^*)$ (similar to $\alpha^p(y^*)$), then $X$ is a $Baire$ space. By Definition \ref{transitive}, the continuous mapping $P^{-1}$ is 
transitive on $X$. It then follows from Lemma \ref{inverse-transitive}, that $P$ is also transitive on $X$. Thus, there exists $x^{''}\in X=\alpha^p(x^*)$ such that $\alpha^p(x^*)=\omega^p(x^{''})$; and moreover, there exists a sequence $\{n_k\}\subset \mathbb{N}$ such that $P^{n_{k}}x^{''}\rightarrow x^*$, as $n_k\to\infty$. For $k\gg 1$, let $x=P^{n_{k}}x^{''}$, then $\omega^{p}(x)=\alpha^{p}(x^*)$.\
Therefore, there exist $x,y$ close enough to $x^*,y^*$, such that $x\in \omega^p(x),x\notin \omega^p(y),y\in \omega^p(y),y\notin \omega^p(x),x-y\in\Lambda $ 
and
\begin{equation}\label{proof of pror3 no.2}
\sigma(x-y)>\sigma(\phi(2n_0T,x)-\phi(2n_0T,y)).
\end{equation}

\par
{\it We now claim that: there exist appropriate $k\in\mathbb{N}$, $\{x^1,x^2,\cdots,x^k\} \in \omega^{p}(x)$, and $\{y^1,y^2,\cdots,y^{k}\}\\ \in \omega^ {p}(y)$, such that
 \begin{align}
\sigma(\phi(0^+,x^{i})-\phi(0^+,y^{i}))&\leq\sigma(\phi(2n_0T,x)-\phi(2n_0T,y)),\quad 1\leq i\leq k,\\
\sigma(x-y)&=\sigma (x^k-y^k).\label{x-y=xk-yk}
 \end{align}}
If the claim is correct, set $i=k$ , then 
$$
\sigma(x-y)\leq\sigma(\phi(2n_0T,x)-\phi(2n_0T,y)),
$$ which is a contradiction to \eqref{proof of 
pror3 no.2}, Proposition \ref{prop3} is then proved.

We now prove that the claim is true. For this purpose, denote
  \begin{equation*}
    \begin{aligned}
    \mathcal{M}=&\{m\in\mathbb{N};\ m\geq\max(N,2n_{0}),\ and\ \phi(mT,x)\in U\},\\
    \mathcal{N}=&\{n\in\mathbb{N};\ n\geq\max(N,2n_{0}),\ and\ \phi(nT,y)\in V\},
    \end{aligned}
  \end{equation*}
where $N,U,V$ is as defined in Lemma $\ref{large-x,y}$.

Note that $x\in\omega^p(x),y\in\omega^p(y)$, then $\mathcal{M}$ and $\mathcal{N}$ are infinite sets. As a consequence, there are natural numbers, 
$k_1,k_2,l_1,l_2$,$m_1,m_2,m_3$,$n_1,n_2,n_3$ as in Lemma $\ref{exist-k}$; and we define sequences as follows:
\begin{equation*}
  \begin{aligned}
  &x^{i}=\phi(m_1T,x),\quad y^{i}=\phi(im_1T,y),\quad 1\leq i \leq k_1,\\
  &x^{k_1+j}=\phi(m_2T,x),\quad y^{k_1+j}=\phi((jm_2+k_1m_1)T,y),
  \quad 1\leq j\leq k_2,\\
  &x^{k_1+k_2+1}=\phi((n_3-k_1m_1-k_2m_2)T,x),\quad y^{k_1+k_2+1}=
        \phi(n_3T,y),\\
    &x^{k_1+k_2+i+1}=\phi((in_{1}+n_3-k_1m_1-k_2m_2)T,x),\quad y^{k_1+k_2+i+1}=
    \phi(n_1T,y),\quad 1\leq i\leq l_1,\\
    &x^{k_1+k_2+l_1+j+1}=\phi((jn_2+l_1n_1+n_3-k_1m_1-k_2m_2)T,x),
    \quad \\&y^{k_1+k_2+l_1+j+1}=
    \phi(n_2T,y),\quad 1\leq j\leq l_2.
  \end{aligned}
\end{equation*}
We now prove the following inequation by induction:
\begin{equation}\label{inequality-1}
  \sigma(\phi(0^+,x^{i})-\phi(0^+,y^{i}))\leq\sigma(\phi(2n_0T,x)-\phi(2n_0T,y)),\, 1\leq i\leq k_1+k_2+l_1+l_2+1.
\end{equation}

(1) If $i=1$, then
\begin{equation*}
  \begin{aligned}
   \sigma(\phi(0^+,x^1)-\phi(0^+,y^1)) &=\sigma(\phi((m_1T)^+,x)-\phi((m_1T)^+,y)\\
&\leq\sigma(\phi(2n_0T,x)-\phi(2n_0T,y)).
  \end{aligned}
  \end{equation*}
Then, \eqref{inequality-1} is correct for $i=1$.

(2) Suppose that there is $i\in \{1,2,\cdots,k_1+k_2-1\},$ such that 
$$ 
\sigma(\phi(0^+,x^i)-\phi(0^+,y^i))\leq\sigma(\phi(2n_0T,x)-\phi(2n_0T,y)).
$$ Note that
\begin{equation*}
\begin{aligned}
 x^{i+1}-y^{i+1}&=\phi(m_1T,x)-\phi(m_1T,y^{i}),\quad 1\leq i<k_1\\
x^{i+1}-y^{i+1}&=\phi(m_{2}T,x)-\phi(m_{2}T,y^{i}),\quad k_1\leq i<k_1+k_2
\end{aligned}
\end{equation*}
Let $x^*=x^{i},y^*=y^{i},m=m_s,s=1,2$, by virtue of \eqref{large-x},
\begin{equation*}
\begin{split}
  \sigma(\phi(0^+,x^{i+1})-\phi(0^+,y^{i+1}))&=\sigma(\phi((m_sT)^+,x)-\phi((m_sT)^+,y^{i}))\\
  &\leq\sigma(\phi(0^+,x^{i})-\phi(0^+,y^{i})\\
  &\leq \sigma(\phi(2n_0T,x)-\phi(2n_0T,y)).
\end{split}
\end{equation*}
Therefore,
\begin{equation*}
  \sigma(\phi(0^+,x^{i})-\phi(0^+,y^{i}))\leq\sigma(\phi(2n_0T,x)-\phi(2n_0T,y)),\quad 1\leq i \leq k_1+k_2.
\end{equation*}

(3) If $i=k_1+k_2+1$, then
\begin{equation*}
 \begin{aligned}
 &x^{k_1+k_2+1}-y^{k_1+k_2+1}=\phi((n_3-k_1m_1-k_2m_2)T,x)-\phi((n_3-k_1m_1-k_2m_2)T,y^{k_1+k_2}).
 \end{aligned}
\end{equation*}
Let $x^*=x^{k_1+k_2},y^*=y^{k_1+k_2},m=n_3-k_1m_1-k_2m_2(m_1>N)$, by \eqref{large-x}
\begin{equation*}
\begin{aligned}
\sigma(\phi(0^+,x^{k_1+k_2+1})-\phi(0^+,y^{k_1+k_2+1}))&=\sigma(\phi((mT)^+,x))-\phi((mT)^+,y^{k_1+k_2})\\
&\leq \sigma(\phi(0^+,x^{k_1+k_2})-\phi(0^+,y^{k_1+k_2})\\
&\leq \sigma(\phi(2n_0T,x)-\phi(2n_0T,y)).
\end{aligned}
\end{equation*}
Thus,
\begin{equation*}
  \sigma(\phi(0^+,x^{i})-\phi(0^+,y^{i}))\leq \sigma(\phi(2n_{0}T,x)-\phi(2n_{0}T,y)),\quad \ i=k_1+k_2+1.
\end{equation*}

\par
Similarly, as to $1\leq i \leq l_1+l_2$, let $x^{'}=x^{i},y^{'}=y^{i},m=n_s,s=1,2$, by \eqref{large-y} we have
\begin{equation*}
  \sigma(\phi(0^+,x^{k_1+k_2+1+i})-\phi(0^+,y^{k_1+k_2+1+i}))\leq \sigma(\phi(2n_0T,x)-\phi(2n_0T,y)).
\end{equation*}
In conclusion,
\begin{equation*}
\sigma(\phi(0^+,x^{i})-\phi(0^+,y^{i}))\leq \sigma(\phi(2n_{0}T,x)-\phi(2n_{0}T,y)),\quad 1\leq i \leq k,
\end{equation*}
where $k=k_1+k_2+l_1+l_2+1.$

By virtue of Lemma $\ref{large-x,y}$,
$$
\sigma(\bar x-\bar y)\equiv constant,
$$ 
for any $(\bar x,\bar y)\in U\times V$. Observe that 
$x^k=\phi(m_3T,x)\in U, y^k=\phi(n_2T,y)\in V$, \eqref{x-y=xk-yk} is correct as well.

We have already checked that the claim is true. The proof of Proposition \ref{prop3} is completed.
\end{proof}

\subsection{Proof of Theorem \ref{main}}

\begin{proof}[Proof of Theorem \ref{main}]
Define a mapping $h$ on $\omega^p(x^0)$ by
$$
\omega^p(x^0)\ni x\mapsto(x_1,x_2)\in\mathbb{R}^2.
$$
The mapping is one-to-one, otherwise, there exist $ x^1,x^2 \in\omega^p(x^0),x^1\neq x^2 $,\ such that
  \begin{equation*}
    h(x^1)=h(x^2),
  \end{equation*}
  that is,
  \begin{equation*}
    x^1-x^2=(0,0,\ast,\cdots)\notin\Lambda,
  \end{equation*}
 which is contradictory to Theorem \ref{changzhixing-th}.
\par
Thus, for any $y\in h( \omega^p(x^0) )\subset \mathbb{R}^2,$ there is only one $ x\in\omega^p(x^0)$ such that $y=h(x).$ Therefore, one can define
\begin{equation*}
  \tilde \varphi(k,y)=\tilde \varphi(k,h(x))=h(\phi(kT,x))=(\phi(kT,x)_1,\phi(kT,x)_2), \quad k\in\mathbb{Z}.
  \end{equation*}
 Let $M= \omega^p(x^0)$, then $h(M)= h( \omega^p(x^0) )\subset \mathbb{R}^2$, we prove that $(h(M),\mathbb{Z},\tilde \varphi)$ is a dynamical system.
 
(1) $\tilde \varphi(k,\cdot)|_{h(M)}$ is a continuous function for any $k\in \mathbb{Z}$.

Suppose on the contrary that $\tilde \varphi(k,\cdot)$ is discontinuous at some $y^*\in h(M)$. Then, there exist $x^*\in M$ and a pair of sequence $\{x^n\}\subset M$, $\{y^n\}\subset h(M)$ such that
\begin{equation*}
\begin{aligned}
&\tilde \varphi(k,y^n)=h(\phi(kT,x^n))=(\phi(kT,x^n)_1, \phi(kT,x^n)_2), \\
&\tilde \varphi(k,y^*) = h(\phi(kT,x^*))=(\phi(kT,x^*)_1,\phi(kT,x^*)_2),
\end{aligned}
\end{equation*}
and 
  $$
  \phi(kT,x^n) \rightarrow \phi(kT,x^*),\quad \tilde \varphi(k,y^n) \nrightarrow y^*.
  $$ 
Therefore, there exists a subsequence $y^{n_l} \rightarrow y^{'}, y^{'} \neq y^{*}.$ While, $ x^{n_l}\to x^*$ is equivalent to that
\[
  \begin{split}
 & (\phi(kT,x^{n_l})_1, \phi(kT,x^{n_l})_2,\cdots)\to (\phi(kT,x^*)_1,\phi(kT,x^*)_2,\cdots)\\
\Rightarrow 
&\tilde \varphi(k, y^{n_l})=(\phi(kT,x^{n_l})_1, \phi(kT,x^{n_l})_2)\to (\phi(kT,x^*)_1,\phi(kT,x^*)_2)\\
&=\tilde \varphi(k,y^*),
  \end{split}
  \]
which is a contradiction to $y'\neq y^*$.

(2) $\tilde\varphi(0,y)=y.$

We have $\tilde\varphi(0,y)=\tilde\varphi(0,h(x))=h(\phi(0,x))=h(x)=y.$

(3) $\tilde\varphi(n_1+n_2,y)=\tilde\varphi(n_1,\tilde\varphi(n_2,y)).$

  Let $\widetilde{x}=\phi(n_2T,x)$, then $\tilde\varphi(n_2,y)=\tilde\varphi(n_2,h(x))=h(\phi(n_2T,x))=h(\widetilde{x})$. Therefore,
\begin{equation*}
 \begin{aligned}
            \tilde\varphi(n_1,\tilde\varphi(n_2,y))&=\tilde\varphi(n_1,h(\widetilde{x}))\\&=h(\phi(n_1T,\widetilde{x}))\\
            &=(\phi(n_1T,\widetilde{x})_1,\phi(n_1T,\widetilde{x})_2)\\
            &=(\phi(n_1T,\phi(n_2T,x))_1,\phi(n_1T,\phi(n_2T,x))_2)\\
            &=(\phi((n_1+n_2)T,x)_1,\phi((n_1+n_2)T,x)_2),
  \end{aligned}
\end{equation*}
\begin{equation*}
\begin{aligned}
\tilde\varphi(n_1+n_2,y)=&h(\phi((n_1+n_2)T,x))\\=&(\phi((n_1+n_2)T,x)_1,\phi((n_1+n_2)T,x)_2),
 \end{aligned}
 \end{equation*}
 that is, $\tilde\varphi(n_1+n_2,y)=\tilde\varphi(n_1,\tilde\varphi(n_2,y))$.
 
We have checked that $(h(M),\mathbb{Z},\tilde\varphi)$ is a dynamical system. By Definition \ref{insert property}, 
  $(M,\mathbb{Z},\phi)$ is embedded into $(h(M),\mathbb{Z},\tilde\varphi)$. 
 
Thus, Theorem \ref{main} is proved.
\end{proof}

\begin{remark}
{\rm
  It is easy to see that $h:\omega^p(x^0)\to h(\omega^p(x^0))$ is a Lipschitz mapping, by using Floquet theory and constructing nested invariant cones, the inverse mapping $h^{-1}$ restricted to $h(\omega^p(x^0))$ is expected to be also Lipschitz(which means that the embedding is Lipschitz), we will prove it in the following work.}
\end{remark}
\section{Discussions}
\subsection{Transform \eqref{system 1.1}+\eqref{assume system 1.1} into \eqref{system 1.1}+\eqref{assume simple system 1.1}}\label{transformation}
Consider \eqref{system 1.1}+\eqref{assume system 1.1} and let 
\begin{equation}
\begin{split}
&\delta_i=\left\{
\begin{aligned}
1,\quad  \text{ if } \frac{\partial f_i}{\partial x_{i-1}}\geq 0  \\
-1, \quad  \text{ if } \frac{\partial f_i}{\partial x_{i-1}}< 0
\end{aligned}
\right.
,\quad i=1,\cdots,n\\
&\Delta=\delta_1\cdots \delta_n.
\end{split}
\end{equation}
Let $y_i=\mu_ix_i$($i=1,\cdots,n$), where
\[
\mu_1=\delta_1,\mu_2=\delta_1\delta_2,\cdots, \mu_n=\delta_1\cdots\delta_n=\Delta.
\]
Then, 
\begin{equation}\label{transform-equa}
  \dot{y}_i=\mu_i f_i(\mu_{i-1}x_{i-1},\mu_{i}x_{i},\mu_{i+1}x_{i+1})\triangleq g_i(y_{i-1},y_{i},y_{i+1}).
\end{equation}
Moreover, 
\begin{itemize}
  \item[{\rm (i)}] if $2\leq i\leq n$,
  \begin{equation*}
  \begin{split}
    & \operatorname{sgn}(\frac{\partial g_i}{\partial y_{i-1}})=\operatorname{sgn}(\mu_i\frac{\partial f_i}{\partial f_{i-1}}\mu_{i-1})=\mu^2_i=1>0, \\
    & \operatorname{sgn}(\frac{\partial g_{i-1}}{\partial y_{i}})=\operatorname{sgn}(\mu_{i-1}\frac{\partial f_{i-1}}{\partial f_{i}}\mu_{i})=\mu^2_i=1>0;
  \end{split} 
  \end{equation*}
  \item[{\rm (ii)}] $i=1$, 
  \begin{equation*}
  \begin{split}
    & \operatorname{sgn}(\frac{\partial g_1}{\partial y_{n}})=\operatorname{sgn}(\mu_1\delta_1\mu_{n})=\Delta=\delta_1\cdots\delta_n, \\
    & \operatorname{sgn}(\frac{\partial g_{n}}{\partial y_{1}})=\operatorname{sgn}(\mu_{n}\delta_1\mu_{1})=\Delta=\delta_1\cdots\delta_n.
  \end{split} 
  \end{equation*}
\end{itemize}
Therefore, then the Jacobi matrix of \eqref{transform-equa} is one of the following forms:
\begin{itemize}
  \item[{\rm (i)}] $\Delta=1$, 
\[
J(g)=\left(\begin{array}{cccc}
* & \geq 0 & & > 0 \\
>0 & * & \geq 0 & \\
& \ddots & \ddots & \geq 0 \\
\geq 0 & & >0 & *
\end{array}\right)
\]
and it is a positive feedback system;
  \item [{\rm (ii)}] $\Delta=-1$,
\[
J(g)=\left(\begin{array}{cccc}
* & \geq 0 & & < 0 \\
>0 & * & \geq 0 & \\
& \ddots & \ddots & \geq 0 \\
\leq 0 & & >0 & *
\end{array}\right)
\]
and it is a negative feedback system.
\end{itemize}

\subsection{Dissipative conditions}
In this subsection, we impose a dissipative condition on system  \eqref{system 
1.1}+\eqref{assume system 1.1} as the following:
\begin{equation*}
\noindent {\bf (H)}\quad f_i(t,x_{i-1},x_i,x_{i+1})x_i<0,\quad \text{for any}\ |x_i|\geq  C,|x_{i\pm 1}|\leq |x_i|,\ i=1,\cdots,n,\ \text{and all}\ t\in[0,T],
\end{equation*}
where $C>0$ is constant. 

Fix $x\in\Omega$ and let $\phi(t,x)$ be the solution of \eqref{system 
1.1}+\eqref{assume system 1.1} with $\phi(0,x)=x$. Then, similar as the arguments in \cite[Lemma 4.3]{WZ}, $\phi(t,x)$ will enter $A=\{x\in\mathbb{R}^n:|x_i|\leq C\ \text{for all }i\}$ at some finite time and then remains there. And hence, we have

\begin{theorem}
  Assume that {\bf (H)} holds true. Then, for any $x\in\Omega$, the  $\omega$-limit set $\omega^p(x)$ of the Poincar\'{e} map $P$ of \eqref{system 
1.1}+\eqref{assume system 1.1} at $x$ can be continuous embedded into a compact subset of $\mathbb{R}^2$.
\end{theorem}

\subsection{$1$-cooperative and $2$-cooperative systems}\label{1-2 cooperative}
As we have known that, the system \eqref{cooperative} is a $2$-cooperative system equivalents to that the linear variation system of \eqref{cooperative}  is of the following form:
\begin{itemize}
  \item  $(-1)a_{1 n}(t),(-1)a_{n 1}(t) \geq 0$;
  \item  $a_{i j}(t) \geq 0$ for all $i, j$ with $|i-j|=1$;
  \item  $a_{i j}(t)=0$ for all $i, j$ with $1<|i-j|<n-1$.
\end{itemize}
Moreover, if there exist $i,j\in \{1,\cdots,n\}$ such that $a_{i-1,i}(t)\equiv a_{j,j+1}(t)\equiv 0$, then by adjusting the order of $x_1,\cdots,x_n$ appropriately, and use a transformation like in Section \ref{transformation}, a  $2$-cooperative system can be converted into a tridiagonal structure with its linear variation system in the following form:
 \begin{itemize}
  \item  $a_{i j}(t) \geq 0$ for all $i, j$ with $|i-j|=1$;
  \item  $a_{i j}(t)=0$ for all $i, j$ with $1<|i-j|<n-1$.
\end{itemize}
Particularly, if $a_{i j}(t) > 0$ for all $i, j$ with $|i-j|=1$, \eqref{cooperative}  is a tridiagonal cooperative system as in \cite{smillie1984,simith1991}, and then a $1$-cooperative system. 

For instance, consider the following four-dimensional system of Lotka-Volterra type:
\begin{equation}\label{no t Lotka-Volterra}
\begin{aligned}
   \dot{x_1}&=x_1[\gamma_1-a_{11}x_1-a_{14}x_4],\\
   \dot{x_2}&=x_2[\gamma_2+a_{21}x_1-a_{22}x_2+a_{23}x_3], \\
   \dot{x_3}&=x_3[\gamma_3+a_{32}x_2-a_{33}x_3+a_{34}x_4], \\
   \dot{x_4}&=x_4[\gamma_4+a_{43}x_3-a_{44}x_4],
\end{aligned}
\end{equation}
where $\gamma_i>0,\ a_{ij}>0,\ x_i(0)\geq0, \ 1\leq i \leq 4 $. Here $x_i$ denotes the density of species $i$ at time $t$, $r_i$ is the intrinsic growth rate of 
species $i$,\ $a_{ij},i\neq j$ measures the action of species $j$ on the growth of species $i$, and $a_{ii}$ represents the crowding effect within species $i$ (see \cite{elkhader1992_sub}).

If $a_{14}=0$, then \eqref{no t Lotka-Volterra} is a tridiagonal cooperative system. Therefore, any bounded solutions of the system converge to an equilibrium point (see \cite{smillie1984}). If $a_{14}<0$, then the system is cooperative, which 
is discussed in the literatures \cite{hirsch1982,hirsch1984,Hofbauer1988,Per,smith1988} and Poincar\'{e}-Bendixson theorem holds true. If $a_{14}>0$, the system is neither cooperative nor competitive but a strongly $2$-cooperative system (see \cite{smith1988, weiss2021_sub}), and Poincar\'{e}-Bendixson theorem also established(see \cite{weiss2021_sub}). In \cite[Lemma 3.1]{elkhader1992_sub}, Elkhader used Hofbauer's conclusion in \cite{Hofbauer1988} to show that the solution of system 
\eqref{no t Lotka-Volterra} is bounded under the condition $a_{22}(a_{33}a_{44}-a_{43}a_{34})-a_{23}a_{32}a_{44}>0$; moreover, under further assumptions, boundary equilibria and interior equilibrium can exist. 

Now consider the time period Lotka-Volterra system:
\begin{equation}\label{Lotka-Volterra}
\begin{aligned}
   \dot{x_1}&=x_1[\gamma_1(t)-a_{11}(t)x_1-a_{14}(t)x_4],\\
   \dot{x_2}&=x_2[\gamma_2(t)+a_{21}(t)x_1-a_{22}(t)x_2+a_{23}(t)x_3], \\
   \dot{x_3}&=x_3[\gamma_3(t)+a_{32}(t)x_2-a_{33}(t)x_3+a_{34}(t)x_4], \\
   \dot{x_4}&=x_4[\gamma_4(t)+a_{43}(t)x_3-a_{44}(t)x_4],
\end{aligned}
\end{equation}
where $(x_1,x_2,x_3,x_4)\in \mathbb{R}^{4}_{+}$,\ $a_{ij}(t)$ and $\gamma_i(t)$ are continuous periodic functions with a minimum positive 
period of $T$, and all other assumptions for \eqref{no t Lotka-Volterra} are also hold for \eqref{Lotka-Volterra}. If $a_{14}(t)\equiv 0$, then \eqref{Lotka-Volterra} is a time $T$-periodic tridiagonal cooperative system; and hence, any bounded solutions of the system converge to a $T$-peridoic orbit(see \cite{simith1991}). If $a_{14}(t)>0$, then it is strongly $2$-cooperative system, and our conclusion Theorem \ref{main} can be used here.

There are some interesting questions for \eqref{Lotka-Volterra}. For instance, whether conditions similar to 
$a_{22}(a_{33}a_{44}-a_{43}a_{34})-a_{23}a_{32}a_{44}>0$ can also imply dissipative of \eqref{Lotka-Volterra}? Whether there exist boundary or interior periodic orbits under further assumptions? We plan to answer these questions in the following work.

\end{document}